\documentclass{article}

\usepackage{selectp}

\PassOptionsToPackage{numbers, compress}{natbib}

\usepackage[final]{neurips_2019}

\usepackage{amssymb,amsmath,amsthm,amsfonts,bbm}
\usepackage{hyperref}
\usepackage{cleveref}

\usepackage{constants}
\usepackage{soul}

\usepackage[utf8]{inputenc} 
\usepackage[T1]{fontenc}    
\usepackage{hyperref}       
\usepackage{url}            
\usepackage{booktabs}       
\usepackage{amsfonts}       
\usepackage{nicefrac}       
\usepackage{microtype}      

\usepackage{thmtools}
\usepackage{thm-restate}

\title{First order expansion of convex regularized estimators}

\author{%
  Pierre C Bellec, \\
  Department of Statistics,\\
  Rutgers University,\\
  501 Hill Center, \\
  Piscataway, NJ 08854, USA. \\
  \texttt{pierre.bellec@rutgers.edu} 
  \And
  Arun K Kuchibhotla, \\
  Department of Statistics,\\
  The Wharton School,\\
  University of Pennsylvania,\\
  Philadelphia, PA 19104, USA.\\
  \texttt{arunku@upenn.edu}
}

\usepackage{mathtools}
\usepackage{thmtools}

\declaretheorem[name=Theorem,numberwithin=section]{theorem}
\declaretheorem[name=Proposition,sibling=theorem]{proposition}
\declaretheorem[name=Lemma,sibling=theorem]{lemma}

\crefname{theorem}{theorem}{theorems}
\crefname{lemma}{lemma}{lemmas}
\crefname{proposition}{proposition}{propositions}
\crefname{assumption}{assumption}{assumptions}
\crefname{example}{example}{examples}
\crefname{corollary}{corollary}{corollaries}

\usepackage[textsize=scriptsize,textwidth=1.3in]{todonotes}

\usepackage{enumitem}

\usepackage{times}

\newcommand{\X}{\mathsf{X}}

\newcommand{\R}{\mathbb{R}}
\newcommand{\E}{\mathbb{E}} 
\renewcommand{\P}{\mathbb{P}} 
\newcommand{\eps}{\varepsilon}
\newcommand{\hbeta}{\hat\beta}
\DeclareMathOperator*{\argmin}{argmin}

\DeclareMathOperator*{\supp}{supp}

\newcommand{\Tr}{\mathsf{trace}}

\newcommand{\prox}{\mathsf{prox}}

\begin{document}

\maketitle

\begin{abstract}

We consider first order expansions of convex penalized estimators in high-dimensional
regression problems with random designs. Our setting
includes linear regression and logistic regression as special cases.
For a given penalty
function $h$ and the corresponding penalized estimator
$\hbeta$, we construct a quantity $\eta$, the first order expansion of $\hbeta$,
such that the distance between $\hbeta$ and $\eta$
is an order of magnitude smaller than the estimation error $\|\hat{\beta} - \beta^*\|$. 
In this sense, the first order expansion $\eta$ can be thought of as a generalization
of influence functions from the mathematical statistics literature
to regularized estimators in high-dimensions.
Such first order expansion implies that the risk of $\hat{\beta}$ is asymptotically the same as the risk of $\eta$ which leads to a precise characterization of the MSE of $\hbeta$;
this characterization takes a particularly simple form for isotropic design.
Such first order expansion also leads to inference results based on $\hat{\beta}$. 
We provide sufficient conditions for the existence of such
first order expansion for three regularizers: the Lasso in its constrained form,
the lasso in its penalized form, and the Group-Lasso.
The results apply to general loss functions under some conditions and those
conditions are satisfied for the squared loss in linear regression
and for the logistic loss in the logistic model.
\end{abstract}

\paragraph{Introduction.}

We consider learning problems where one observes observations $(X_1,Y_1),...,(X_n,Y_n)$
with responses $Y_i$ and feature vectors $X_i\in\R^p$. The literature of the past
two decades has demonstrated the great success of regularized estimators
that are commonly defined as solutions to regularized optimization problems
of the form
\begin{equation}
    \textstyle
    \label{hbeta-intro}
    \hbeta = \argmin_{\beta\in\R^p}
    n^{-1} \sum_{i=1}^n \ell(Y_i,X_i^T\beta) + h(\beta),
\end{equation}
where $\ell(\cdot,\cdot)$ is referred to as the loss (e.g. squared loss, logistic loss)
and $h:\R^p\to \R$ is a regularization penalty (e.g. the $\ell_1$-norm for the Lasso, the $\ell_{2,1}$ norm for the Group-Lasso).
All tuning parameters are included in $h(\cdot)$.
The performance of such regularized estimators is measured in terms of prediction
error or in terms estimation error $\|\hbeta-\beta^*\|$
if the data comes from a model such as $Y=X\beta^* + \eps$ for some noise random
variable $\eps$ in linear regression
or $$\P(Y=1|X=x) = 1/(1+\exp(x^T\beta^*)) = 1- \P(Y=0|X=x)$$
in logistic regression, where $\beta^*$ is the unknown coefficient vector.
For instance, if $s=\|\beta^*\|_0$ is the sparsity of $\beta^*$ in the 
above model, and $(X_i,Y_i)_{i=1,...,n}$ are iid observations with the same distribution as $(X,Y)$,
both the Lasso in linear regression and the logistic Lasso
in logistic regression enjoy rate optimality:
$\|\hbeta-\beta^*\|^2 \le s\log(ep/s)/n$; see \cite{negahban2009unified,alquier2019estimation}
or the proof of \Cref{prop:RateLogisticLasso-mainpaper} in \Cref{sec:RSCLogistic} for self-contained proofs.
The latter estimation bound is optimal in a minimax sense and cannot be improved,
and the minimax rate $s\log(ep/s)/n$ represents the scale below which
uncertainty is unavoidable by information theoretic arguments, see for instance
\cite[Section 5]{rigollet2011exponential}.

We are interested in providing first order expansion of $\hbeta$ at scales negligible
compared to the minimax estimation rate, e.g. at scales negligible compared
to $s\log(ep/s)/n$ in the aforementioned sparsity contexts.
To be more precise, the results below will construct random first order expansion
$\eta$
such that $\eta$ is measurable w.r.t. a much smaller sigma algebra than
that generated by $(X_i,Y_i)_{i=1,...,n}$, and
\begin{equation}
    \label{approximation-intro}
    \|\eta-\hbeta\|_K^2 = o_p(1) \|\hbeta-\beta^*\|_K^2
    \quad
    \text{ for some norm $\|\cdot\|_K$ related to the problem at hand,}
\end{equation}
where $o_p(1)$ is a quantity that converges to 0 in probability.
In other words, we provide a first-order expansion of $\hat{\beta}$ similar to an influence function expansion, cf. \Cref{sec:influence}. This allows for understanding bias and standard deviation of $\hat{\beta}$ at a finer scale 
than simply showing that $\hbeta-\beta^*$ converge to zero at the minimax rate.
The present paper intends to answer the two questions below regarding such first order expansion.
\begin{enumerate}[label=\bfseries (Q\arabic*)]
    \item How to construct $\eta$ such that \eqref{approximation-intro} holds
        for a given convex regularized estimator
        such as \eqref{hbeta-intro}?\label{eq:Q1}
    \item How are such first order expansions useful in high-dimensional learning problems
        where convex regularized estimators \eqref{hbeta-intro} are commonly used?
        \label{eq:Q2}
\end{enumerate}
An expansion $\eta$ satisfying \eqref{approximation-intro} is interesting
in and by itself because it describes phenomena at a finer scale than
most of the literature in high-dimensional problems which focuses on
minimax prediction and estimation bounds.
More importantly, we will see in Section~\ref{sec:ExactRiskBound} that such first-order
expansions
lead to
exact identities for the loss of estimators, and in
Section~\ref{sec:Inference} that such first-order expansions
can be used for inference (i.e., uncertainty quantification) about the
unknown coefficient vector $\beta^*$.

\paragraph{Notation.} Throughout the paper, $\C,\C,\C,...$ denote positive absolute
constants and we write $a \lesssim b$ if $a\le C b$ for some absolute constant $C>0$.
The Euclidean norm in $\mathbb{R}^p$ or in $\R^n$ is denoted by $\|\cdot\|$. For any positive definite matrix $A$, we write $\|u\|_A = \|A^{1/2}u\|$ for the matrix square-root $A$.
For matrices, $\|\cdot\|_{op}$ and $\|\cdot\|_{F}$ denote the operator norm
and Frobenius norm.
For any real $a$, $a_+ = \max(0,a)$. 
If $S\subset \{1,...,p\}$, 
$v\in\R^p, M\in\R^{p\times p}$ 
then $v_S$ is the restriction $(v_j, j\in S)$ and $M_{S,S}$ is the square
submatrix of $M$ made of entries indexed in $S\times S$.

\section[Influence functions and Intuition]{Influence functions and Construction of $\eta$}
\label{sec:influence}
To answer~\ref{eq:Q1}, we start with a recap of unregularized estimators that correspond to $h(\cdot)\equiv 0$, when $p$ is fixed as $n\to+\infty$.
In this case, it is well-known that certain smoothness assumptions on the loss such as twice differentiability~\citep{2018arXiv180905172K,ichimura1993semiparametric} or stochastic equicontinuity~\citep{van2000asymptotic,van2002part} imply (for any norm, since
all norms are equivalent in $\R^p$ for fixed $p$):
\begin{equation}\label{eq:UnregRep}
\left\|\hat{\beta} - \beta^* - \sum_{i=1}^n \frac{\psi(X_i, Y_i)}{n}\right\| = o_p(1)\|\hat{\beta} - \beta^*\|\Leftrightarrow(1 + o_p(1))\sqrt{n}(\hat{\beta} - \beta^*) = \sum_{i=1}^n \frac{\psi(X_i,Y_i)}{\sqrt{n}},
\end{equation}
for some target $\beta^*$ and a mean zero function $\psi(\cdot, \cdot)$ sometimes referred to as the influence function. 
See~\citep[Theorem 3.1]{2018arXiv180905172K}, \citep[Page 52]{van2000asymptotic},~\citep[Theorem 6.17]{van2002part},~\citep[Lemma 5.4]{ichimura1993semiparametric} for details.
In this case we can take $\eta = \beta^* + \sum \psi(X_i, Y_i)/n$ in~\eqref{approximation-intro}.
This representation allows us to claim asymptotic unbiasedness and fluctuations of order $n^{-1/2}$ for $\hat{\beta}$ around $\beta^*$.
It also shows that estimator $\hat{\beta}$ behaves like an average and hence allows transfer of results (e.g., central limit theorems) for averages to study of $\hat{\beta}$ in terms of variance estimation, confidence intervals, hypothesis testing and bootstrap.

A general study of such representation for regularized problems is lacking in the literature. 
\cite{knight2000asymptotics} is the first work that analyzed linear regression Lasso when the number of covariates $p$ is fixed and does not change with the sample size $n$. 
In the more challenging regime where $p \ge n$,
Theorem 5.1 of~\cite{javanmard2015biasing} provides a first order expansion allowing for $p$ to diverge (almost exponentially) with $n$. 
In the present work, we simplify and present a unified derivation of such first order expansion result,
generalizing~\cite[Theorem 5.1]{javanmard2015biasing} beyond the squared loss,
beyond the $\ell_1$ penalty and beyond certain assumptions of \cite{javanmard2015biasing}
on $\E[X_i X_i^T]$.
The derivation of~\eqref{eq:UnregRep} can be motivated by defining
\begin{equation}\label{eq:EtaTildeDef}
\tilde{\eta} := \argmin_{\beta\in\mathbb{R}^p}\,\frac{1}{n}\sum_{i=1}^n \ell'(Y_i, X_i^{\top}\beta^*)X_i^{\top}(\beta - \beta^*) + \frac{1}{2n}\sum_{i=1}^n \ell''(Y_i, X_i^{\top}\beta^*)\{X_i^{\top}(\beta - \beta^*)\}^2 + h(\beta),
\end{equation}
with $h(\cdot)\equiv 0$. 
Here and throughtout $\ell'(y, u)$ and $\ell''(y, u)$ represent (first and second) partial derivatives of $\ell$ with respect to $u$.
The right hand side of~\eqref{eq:EtaTildeDef} (with $h(\cdot)\equiv 0$) is the quadratic approximation of $\sum_{i=1}^n \ell(Y_i, X_i^{\top}\beta)/n$ around $\beta = \beta^*$ (without the term independent of $\beta$). The final first order expansion $\eta$ is obtained by replacing the quadratic part of the approximation by its expectation as in the next display.
Following the intuitive construction of $\eta$ for the unregularized problem, we construct a first order expansion for the regularized problem 
as
\begin{equation}\label{eq:FirstOrderApproxRegularized}
\begin{split}
\eta &:= 
\textstyle
\argmin_{\beta\in\mathbb{R}^p}\,n^{-1}\sum_{i=1}^n \ell'(Y_i, X_i^{\top}\beta^*)X_i^{\top}(\beta - \beta^*) + \frac{1}{2}(\beta - \beta^*)^{\top}K(\beta - \beta^*) + h(\beta)\\
&:=
\textstyle
\argmin_{\beta\in\mathbb{R}^p}\,\frac{1}{2}\left\|K^{1/2}\left(\beta - \beta^* - n^{-1}\sum_{i=1}^n K^{-1}X_i\ell'(Y_i, X_i^{\top}\beta^*)\right)\right\|^2 + h(\beta).
\end{split}
\end{equation}
where
$K := n^{-1}\sum_{i=1}^n \mathbb{E}\left[\ell''(Y_i, X_i^{\top}\beta^*)X_iX_i^{\top}\right].$
From this definition, we can write $\eta = \eta_K\left(\beta^* + n^{-1}\sum_{i=1}^n K^{-1}X_i\ell'(Y_i, X_i^{\top}\beta^*)\right),$
for a function $\eta_K(\cdot)$ (depending on $h(\cdot), K$).
Our main results prove under some mild assumptions that
\[
    \textstyle
\|\hat{\beta} - \eta_K(\beta^* + \frac{1}{n}\sum_{i=1}^n \ell'(Y_i, X_i^{\top}\beta^*)X_i)\|_K ~=~ o_p(1)\|\hat{\beta} - \beta^*\|_K.
\]
Comparing this with~\eqref{eq:UnregRep} we note that for the unregularized problem, $\eta_K(\beta) = \beta$ is the identity.

\section{Main Results: Approximation Theorem}
\label{sec:main-approximation-theorem}


We introduce the notion of Gaussian complexity for the following results. For any set $T\subset\mathbb{R}^p$ and a covariance matrix $\Sigma$, the Gaussian complexity of $T$ is given by
\begin{equation}\label{def:mean-width}
    \textstyle
\gamma(T, \Sigma) := \mathbb{E}\left[\sup_{u\in T:\,\|\Sigma^{1/2}u\| = 1}\,|g^{\top}\Sigma^{1/2}u|\right] = \mathbb{E}\left[\sup_{u\in T}\,\frac{|g^{\top}\Sigma^{1/2}u|}{\|\Sigma^{1/2}u\|}\right],
\end{equation}
where the expectation is with respect to the standard normal vector $g\sim N(0,I_p)$. We also need the notion of $L$-subGaussianity. A random vector $X$ is said to be $L$-subGaussian with respect to a (positive definite) matrix $\Sigma$ if 
\begin{equation}\label{assum:subgaussian}
    \forall u\in\R^p, \quad \E[\exp(u^T X)] \le \exp(L^2\|u\|_\Sigma^2/2)
    \qquad
    \text{ where }
    \|u\|_\Sigma = \|\Sigma^{1/2} u\|.
\end{equation}
This implies
$\sup_u\mathbb{P}(|u^{\top}X| \ge t\|u\|_{\Sigma}) \le 2\exp(-{t^2}/{(2L^2)})$.
Recall that the scaled norm $\|\cdot\|_{K}$ is defined by
$\|u\|_{K}^2 = n^{-1}\sum_{i=1}^n \E[ \ell''(Y_i,X_i^\top\beta^*)(X_i^\top u)^2].
$
Consider the following assumptions:
\begin{enumerate}[label=\bfseries (A\arabic*)]
\item\label{eq:LossAssump} There exists constants $0 \le B, B_2, B_3 < \infty$ such that the loss satisfies $\forall u_1,u_2\in\R, \forall y$,
\begin{equation}\label{eq:ProcessControlAssumption}
    \frac{|\ell''(y,u_1)-\ell''(y,u_2)|}{|u_1 - u_2|} \le B,\qquad
|\ell''(y, u_1)| \le B_2,
\qquad \sup_{u\in \mathbb{R}^p}\frac{\|\Sigma^{1/2}u\|^2}{\|K^{1/2}u\|^2} \le B_3 .
\end{equation}
\item\label{eq:SubGaussianAssump} The observations $(X_1, Y_1), \ldots, (X_n, Y_n)$ are iid.
    Further $X_1, \ldots, X_n$ are mean zero and $L$-subGaussian with respect to their covariance $\Sigma$, i.e., \eqref{assum:subgaussian} holds.
\end{enumerate}
Note that $L$ in \ref{eq:SubGaussianAssump} is necessarily no smaller than one, i.e.,
$L\ge1$.
Define the error
\begin{equation}\label{eq:EstimationError}
\mathcal{E} := \|\hat{\beta} - \beta^*\|_K + \|\eta - \beta^*\|_K
\qquad\text{ where \qquad} \|\cdot\|_K
\text{ is the norm } \|u\|_K = \|K^{1/2} u\|.
\end{equation}
The quantity $\mathcal{E}$ quantifies the error made by
$\hbeta$ and $\eta$ in estimating $\beta^*$ with respect to the norm $\|\cdot\|_K$.
Bounds on $\|\hbeta - \beta^*\|_K$ and $\|\eta - \beta^*\|_K$ follow from the existing literature; see~\cite{negahban2009unified} or~\Cref{prop:RateLogisticLasso-mainpaper}
and its proof in \Cref{sec:RSCLogistic}.
\begin{theorem}
    \label{thm:main}
    Let $r_n := n^{-1/2}\gamma(T,\Sigma)$ and assume that $r_n\le 1$.
    Further assume~\ref{eq:LossAssump} and~\ref{eq:SubGaussianAssump} hold true.
Then with probability at least $1-2 e^{-\C n r_n^2} - 2e^{-\C\log n}$ we have the following: 
\begin{enumerate}
    \item If $\{\hat{\beta} - \beta^*,\eta - \beta^*\}\subseteq T$ then 
    $\|\hbeta - \eta\|_K \lesssim  LB_2 B_3 r_n^{1/2}\mathcal{E} + B^{1/2} (B_3L)^{3/2}  ( 1 + r_n^3 \sqrt n)\mathcal{E}^{3/2}$.
    \item If $\{\hbeta-\eta,\hbeta - \beta^*,\eta - \beta^*\}\subseteq T$ then
    $\|\hbeta - \eta\|_K \lesssim  B_2 B_3L^2 r_n \mathcal{E} + B B_3^{3/2}L^3 (1+r_n^3 \sqrt n) \mathcal{E}^2$.
\end{enumerate}
\end{theorem}
The set $T$ mentioned in Theorem~\ref{thm:main}(1)
are available in the literature for many convex 
penalties. In the following, we
will find this for
the constrained Lasso, penalized Lasso, and Group-Lasso
(with non-overlapping groups) under sharp conditions.
We refer to~\cite{bellec2016slope} for slope penalty,
and~\citet[Lemma 1]{negahban2009unified} and
~\citet[Def. 4.4 and Theorem 4.1]{van2014weakly}
where set $T$ is
presented for a general class of penalty functions 
including nuclear norm, Group-Lasso (with
overlapping groups). 

Proofs of \Cref{thm:main} and all following results are given in the supplement.
An outline of \Cref{thm:main} is given in \Cref{sec:proof-sketch}.
Although Theorem~\ref{thm:main} is stated under assumption~\ref{eq:SubGaussianAssump}, we present a deterministic version of the result (in Section~\ref{sec:proof-sketch}) that replaces $r_n$ by suprema of different stochastic processes.

\paragraph{Squared loss in the linear model.} 
Consider $\ell(y,u) = (y - u)^2/2$ and $n$ iid observations
\begin{equation}
    \label{eq:linear-model}
    \textstyle
    Y_i = X_i^T\beta^* + \eps_i, \;
    \text{and $X_i$ is independent of $\varepsilon_i$ for}\; i=1,\ldots,n,
\end{equation}
Then we have $K= \Sigma= \E[X_1X_1^T]$ and the second derivative $\ell''$ is
constant.
Hence condition \eqref{eq:ProcessControlAssumption} is satisfied with $B=0$
and $B_2=B_3 = 1$. The conclusions of the Theorem~\ref{thm:main} can be rewritten as
\begin{align}
    \{\hbeta - \beta^*,\eta - \beta^*\}\subseteq T 
    \quad&\Rightarrow\quad
    \|\hbeta - \eta\|_K ~\lesssim~ Lr_n^{1/2}\mathcal{E}.
    \label{eq:square-loss-slow-rate}
    \\
    \{\hbeta-\eta,\hbeta - \beta^*,\eta - \beta^*\}\subseteq T
    \quad&\Rightarrow\quad
    \|\hbeta - \eta\|_K ~\lesssim~ L^2r_n \mathcal{E},
    \label{eq:square-loss-fast-rate}
\end{align}
where $\mathcal{E} = \|\Sigma^{1/2}(\hbeta-\beta^*)\| + \|\Sigma^{1/2}(\eta-\beta^*)\|$.
Since $r_n\le1$ (and typically $r_n\to 0$ while $L$ stays bounded, as we will see in the examples below),
the inequality in \eqref{eq:square-loss-fast-rate} is stronger
than the inequality in \eqref{eq:square-loss-slow-rate}.
In the linear model, we thus refer to inequality \eqref{eq:square-loss-slow-rate}
as the ``slow rate'' inequality, and to \eqref{eq:square-loss-fast-rate} 
as the ``fast rate'' one.
The set $T$ encodes
the low-dimensional structure and characterizes the rate $r_n$ through the Gaussian
complexity $\gamma(T,\Sigma)$.
The fast rate inequality is granted provided that $T$
contains the difference $(\eta-\hbeta)$ additionally to the error vectors 
$\{\hbeta - \beta^*,\eta - \beta^*\}$. Conditions that ensure the fast rate inequality will be
made explicit in \Cref{sec:lasso-penalized,subsec:GL} for the Lasso and Group-Lasso.

\paragraph{Logistic loss in the logistic model.} The following proposition shows that \eqref{eq:ProcessControlAssumption} is again satisfied.
\begin{restatable}{proposition}{propositionLogisticSetting}
    \label{prop:logistic-setting}
    Consider the logistic loss $\ell(y,u) = yu - \log(1+e^{u})$ for $y\in\{0,1\}, u\in\R$.
    Assume that $(X_i,Y_i)_{i=1,...,n}$ are iid satisfying the logistic regression model
    $$
    \P(Y_i=1|X_i) = 1-\P(Y_i=0|X_i)= 1/(1+\exp({X_i{}^\top\beta^*})),$$
    for some $\beta^*\in\R^p$ with $\|\Sigma^{1/2}\beta^*\| \le 1$.\footnote{The constant $1$ can be replaced by another absolute constant; this will only change $B_3$ to a different constant.}
    Assume~\ref{eq:SubGaussianAssump} holds.
    Then~\eqref{eq:ProcessControlAssumption} holds with $B=1/(6\sqrt 3)$, $B_2
    = 1$ and an absolute constant $B_3> 0$.

\end{restatable}
In this logistic model, the conclusions of \Cref{thm:main} present an extra term
compared to the linear model with squared loss because
the Lipschitz constant $B$ in \eqref{eq:ProcessControlAssumption} is non-zero:
\Cref{thm:main} reads that with high probability
\begin{align}
    \{\hbeta -\beta^*,\eta - \beta^*\}\subset T 
    \quad&\Rightarrow\quad
    \|\hbeta - \eta\|_K \lesssim Lr_n^{1/2}\mathcal{E} +  B^{1/2}L^{3/2}( 1 + r_n^3 \sqrt n)\mathcal{E}^{3/2},
    \label{eq:slow-rate-logistic}\\
    \{\hbeta-\eta,\hbeta - \beta^*,\eta - \beta^*\}\subset T
    \quad&\Rightarrow\quad
    \|\hbeta - \eta\|_K \lesssim L^2r_n \mathcal{E} + BL^3(1+r_n^3 \sqrt n) \mathcal{E}^2.\label{eq:fast-rate-logistic}
\end{align}
Similar to the case of squared loss, inequality~\eqref{eq:fast-rate-logistic} is stronger than inequality~\eqref{eq:slow-rate-logistic} when $\hbeta - \eta$ belongs in $T$
additionally to $\{\hbeta-\beta^*,\eta-\beta^*\}\subset T$.

For the squared loss, \Cref{table:squared} in the supplement
summarizes the tuning parameters,
minimax rates for $\|\hat\beta-\beta^*\|_\Sigma$, Gaussian width bounds, and
upper bounds  \eqref{eq:square-loss-slow-rate}-\eqref{eq:square-loss-fast-rate}
for the Lasso and Group-Lasso.
\Cref{table:logistic} summarizes implications of \eqref{eq:slow-rate-logistic}-\eqref{eq:fast-rate-logistic} for logistic Lasso and logistic Group-Lasso.

\section[Low dimensionality of T]{What is the low-dimensional set $T$? Application to Lasso and Group-Lasso}
\label{sec:4-examples-of-cones-T}
We now provide applications of the above result to three different penalty functions
commonly used in high-dimensional settings. Throughout this section,
for any cone $T\subseteq\mathbb{R}^p$,
let $\phi(T)$ be the smallest singular value of $\Sigma^{1/2}$
restricted
to $T$,  i.e., 
$\phi(T) = \min_{u\in T:\|u\|=1} \|\Sigma^{1/2}u\|$.
Further consider
\begin{enumerate}[label=\bfseries (N\arabic*)]
\item The features are normalized such that $\Sigma_{jj} \le 1$ for all $1\le j\le p$.\label{eq:Normalized}
\end{enumerate}

\subsection{Constrained Lasso}
Let $R>0$ be a fixed parameter.
Our first example studies the constrained Lasso penalty \cite{tibshirani1996regression}
\begin{equation}
    \label{eq:penalty-constrained-lasso}
    h(\beta) = +\infty \quad \text{ if } \|\beta\|_1 > R \qquad \text{ and } 
    \qquad h(\beta)=0  \quad \text { if } \|\beta\|_1\le R, 
\end{equation}
i.e., $h$ is the convex indicator function
of the $\ell_1$-ball of radius $R>0$.
Applying the above result requires two ingredients: proving that the error vectors
$\{\hbeta-\beta^*,\eta-\beta^*\}$ belong to some set $T$ with high probability,
and proving that $r_n=n^{-1/2}\gamma(T,\Sigma)$ is small. Define for any real $k\ge1$,
\begin{equation}
    \label{T_lasso}
    T_{\texttt{lasso}}(k) := \{u\in\mathbb{R}^p:\,\|u\|_1 \le \sqrt{k}\|u\|\}.
\end{equation}
The parameter $k$ above will typically be a constant times $s=\|\beta^*\|_0$,
the sparsity of $\beta^*$.
If $R=\|\beta^*\|_1$, then
the triangle inequality reveals that the error vectors of $\hbeta$ and $\eta$ satisfy
\begin{equation}
    \label{set-T-constrained-lasso}
    \{ \hbeta-\beta^*,\eta-\beta^*\}\subseteq T := \{u\in\R^p: \|u_{S^c}\|_1 \le \|u_S\|_1\},
\end{equation}
where $S=\{j=1,...,p: \beta^*_j \ne 0\}$ is the support of the true $\beta^*$
and $u_S$ is the restriction of $u$ to $S$.
By the Cauchy-Schwarz inequality $\|u_S\|_1\le\sqrt{s}\|u_S\|_2$, thus $T$ in \eqref{set-T-constrained-lasso}
satisfies $T \subset T_{\texttt{lasso}}(4s)$.

\begin{restatable}{lemma}{lemmaUpperBoundConeTLasso}
    \label{lemma:upper-bound-cone-T-lasso}
    If~\ref{eq:Normalized} holds and $k\ge 1$,
    then we have
    $\gamma(T,\Sigma) \lesssim \phi(T)^{-1} \sqrt{k\log(2p/k)}$ for any cone $T\subset T_{\texttt{lasso}}(k)$
    where $T_{\texttt{lasso}}(k)$ is defined in \eqref{T_lasso}.
\end{restatable}
Hence under~\ref{eq:Normalized} and by setting $k=4s$ and
$r_n=\phi(T)^{-1} \sqrt{s\log(ep/s)/n}$ we have in the linear model with squared loss
that, with high probability,
\begin{equation}
    \label{eq:constrained-lasso-upper}
    \|\Sigma^{1/2}(\eta-\hbeta)\|\lesssim L\phi(T)^{-1/2} (s\log(ep/s)/n)^{1/4}(\|\Sigma^{1/2}(\hbeta - \beta^*)\| + \|\Sigma^{1/2}(\eta - \beta^*)\|)
\end{equation}
and we have established that $\eta$ is a first order expansion of $\hbeta$
with respect to the norm $\|\cdot\|_\Sigma$ if $s\log(ep/s)/n\to 0$.
It is informative to study the order of magnitude of the right hand 
side in \eqref{eq:constrained-lasso-upper}.
For that purpose, the following Lemma gives explicit bounds
on $\|\Sigma^{1/2}(\hbeta - \beta^*)\|$ and $\|\Sigma^{1/2}(\eta - \beta^*)\|$.

\begin{restatable}{lemma}{lemmaRiskConstrainedLasso}
    \label{lemma:risk-constrained-lasso}
    Consider the linear model with squared loss \eqref{eq:linear-model}
    and assume~\ref{eq:SubGaussianAssump}.
    Let $\hbeta,\eta$ in \eqref{hbeta-intro} and \eqref{eq:FirstOrderApproxRegularized}
    with penalty \eqref{eq:penalty-constrained-lasso}.
    Then if $R = \|\beta^*\|_1$, 
    we have
    with probability at least $1-2e^{-n r_n^2 }$,
    \begin{equation}
    \|\Sigma^{1/2}(\eta - \beta^*)\| ~\lesssim~ L\sigma^* r_n, \quad\mbox{and}\quad
        \|\Sigma^{1/2}(\hbeta - \beta^*)\| ~\lesssim~ L\sigma^* r_n(1-\C L^2 r_n)^{-1},
        \label{risk-constrained-lasso}
    \end{equation}
    where $r_n=\phi(T)^{-1}\sqrt{s\log(ep/s)/n}$ and $(\sigma^*)^2 = (\eps_1^2+...+\eps_n^2)/n.$
\end{restatable}
The above lemma provides a slight improvement in the rate compared to \cite[Theorem 11.1(a)]{hastie2015statistical}.
Combined with inequality \eqref{eq:constrained-lasso-upper}, we have established that
$\|\Sigma^{1/2}(\hbeta-\eta)\| \lesssim L^2\sigma^* r_n^{3/2}$. If $r_n\to 0$ (e.g., if $s\log(ep/s)/n\to 0$ while
$\phi(T)$ stays bounded away from 0), this means that the distance $\|\Sigma^{1/2}(\hbeta-\eta)\|$
between $\hbeta$ and $\eta$ is an order
of magnitude smaller than the risk bounds in \eqref{risk-constrained-lasso}.

Inclusion \eqref{set-T-constrained-lasso} is granted regardless of the loss $\ell$,
as soon as $\beta^*$ lies on the boundary of  $\{\beta\in\R^p:\|\beta\|_1=R\}$.
In logistic regression, i.e., the setting of \Cref{prop:logistic-setting} with the constrained Lasso penalty \eqref{eq:penalty-constrained-lasso},
inequality \eqref{eq:slow-rate-logistic} yields that with high probability,
$
\|\eta-\hbeta\|_K\lesssim L[r_n^{1/2} + L^{1/2}(1+r_n^3\sqrt n)\mathcal{E}^{1/2} ]
\mathcal{E}
$.
An extra term appears compared to the squared loss. In order
to obtain a first-order expansion as in \eqref{approximation-intro}
requires $r_n\to 0$ as well as $(1+r_n^3\sqrt n)\mathcal{E}^{1/2} \to 0$.
These conditions can be obtained if risk bounds such as \eqref{risk-constrained-lasso}
are available, see \cite{negahban2009unified,alquier2019estimation} or \Cref{prop:RateLogisticLasso-mainpaper} and its proof in \Cref{sec:RSCLogistic}
for applicable general techniques.
A more detailed discussion of Logistic Lasso is given in the next subsection.

\subsection{Penalized Lasso}
\label{sec:lasso-penalized}

We now consider the $\ell_1$-norm penalty 
\begin{equation}
    \label{h-penalized-lasso}
 h(\beta)=\lambda \|\beta\|_1 \qquad \text{
 for some }\qquad \lambda\ge 0.    
\end{equation}

Here, the fact that $\hbeta-\beta^*,\eta-\beta^*\in T$ for some low-dimensional cone $T$
is not granted almost surely, in that regard the situation differs
from the constrained Lasso case in 
\eqref{set-T-constrained-lasso}. We may find such low-dimensional cone $T$
simultaneously for $\hbeta,\eta$ for both the squared loss and logistic loss as follows,
using ideas from \cite{negahban2009unified,bickel2009simultaneous}.
Let $f_n$ be the convex function so that the objective in \eqref{hbeta-intro}
is equal to $f_n(\beta)+ h(\beta)$ and let $g_n$ be the convex function so that
the objective in \eqref{eq:FirstOrderApproxRegularized} is $g_n(\beta) + h(\beta)$.
Since $\hbeta$ and $\eta$ are solutions of the corresponding optimization problems
\eqref{hbeta-intro} and \eqref{eq:FirstOrderApproxRegularized},
\begin{equation}
    \label{optimality-additive-penalty}
    \begin{split}
    h(\hbeta) - h(\beta^*) &\le f_n(\beta^*) - f_n(\hbeta) \le \nabla f_n(\beta^*)^T(\beta^*-\hbeta),
                         \\ 
    h(\eta) - h(\beta^*) &\le g_n(\beta^*) - g_n(\eta) \le \nabla g_n(\beta^*)^T(\beta^* -\eta).
    \end{split}
\end{equation}
Since $\nabla g_n(\beta^*)=\nabla f_n(\beta^*)$, both $\eta$ and $\hbeta$ belong to the
set $\hat T = \{ b\in\R^p: h(b) - h(\beta^*) \le \nabla f_n(\beta^*)^T(b - \beta^*)\}$.
Next, for both the squared loss and the logistic loss,
$\nabla f_n(\beta^*)$ has subGaussian coordinates under~\ref{eq:SubGaussianAssump}.
Combining these remarks, we obtain the following, proved in supplement.

\begin{restatable}{lemma}{lemmaPenalizedLassoBelongsToCone}
    \label{lemma:penalizedLassoBelongstoCone}
    Let $h$ be as in \eqref{h-penalized-lasso}.
    Consider the linear model \eqref{eq:linear-model} and assume~\ref{eq:SubGaussianAssump},~\ref{eq:Normalized}.
    Let $\xi>0$ be a constant and let $\lambda = L \sigma^*(1+3\xi)\sqrt{2\log(p/s)/n}$ where
    $(\sigma^*)^2=(\eps_1^2+\ldots+\eps_n^2)/n$ and $\|\beta^*\|_0=s$. Then
    \begin{equation}
        \label{eq:cone-lasso-penalized}
        \P\left[
        \{\hbeta-\beta^*,\eta-\beta^*\} \subset
        T
        \right]
        \ge
        1 - \frac{2}{\xi^2\log(p/s)(p/s)^\xi}
        \text{ where }
        T =  T_{\texttt{lasso}}\left(s(6+2\xi^{-1})^2\right).
    \end{equation}
    If instead the logistic regression model and assumptions of
    \Cref{prop:logistic-setting} are fulfilled and 
    $\lambda = (L/2)(1+3\xi)\sqrt{2\log(p/s)/n}$, then the previous display~\eqref{eq:cone-lasso-penalized} also holds.
\end{restatable}
The set $T$ above is the set $T_{\texttt{lasso}}(k)$ in \eqref{T_lasso} with $k=s(6+2\xi^{-1})^2$.
Eq. \eqref{eq:cone-lasso-penalized} defines a low-dimensional cone $T$ that contains
both error vectors $\hbeta-\beta^*,\eta-\beta^*$ for the squared loss and the logistic loss.
The Gaussian width of the set $T$ in \eqref{eq:constrained-lasso-upper}
is already bounded in \Cref{lemma:upper-bound-cone-T-lasso}.
Hence the Gaussian width of $T$ in the previous lemma is bounded
from above as in the previous section,
i.e., $\gamma(\Sigma,T)\lesssim\phi(T)^{-1}(6+2\xi^{-1})\sqrt{s\log(2p/s)}$
    by \Cref{lemma:upper-bound-cone-T-lasso},
and the ``slow rate'' inequality
\eqref{eq:constrained-lasso-upper} again holds with high probability,
where $\phi(T)$ denotes the restricted eigenvalue of the set $T$
of the previous lemma.
Risk bounds similar to \eqref{risk-constrained-lasso} are given below.
We emphasize here the fact that the error vectors
of the Lasso belong to the cone \eqref{eq:cone-lasso-penalized}
with high probability is not new: this is a powerful technique used throughout
the literature on high-dimensional statistics starting from \cite{bickel2009simultaneous,negahban2009unified}.
The novelty of our results
are inequalities such as \eqref{eq:constrained-lasso-upper} which shows
that the distance $\|\Sigma^{1/2}(\hbeta-\eta)\|$ is an order of magnitude
faster than the minimax risk $\sqrt{s\log(ep/s)/n}$. 
\Cref{table:squared} in the supplement
summarizes the tuning parameters,
minimax rates for $\|\hat\beta-\beta^*\|_\Sigma$, Gaussian width bound and
upper bounds  on $\|\hat\beta-\eta\|_\Sigma$.

We will now state a result similar to~\Cref{lemma:risk-constrained-lasso} for linear and logistic Lasso.

\begin{restatable}{proposition}{propRateLogisticLasso}
\label{prop:RateLogisticLasso-mainpaper}
Consider the penalized Lasso estimator $\hat{\beta}$ given by
\[\textstyle
\hat{\beta} := \argmin_{\beta\in\mathbb{R}^p}\,\frac{1}{n}\sum_{i=1}^n \ell(Y_i, X_i^{\top}\beta) + \lambda\|\beta\|_1,
\] 
where $\ell$ is either the squared or logistic loss 
and 
$\lambda$ is chosen as in~\Cref{lemma:penalizedLassoBelongstoCone} for some $\xi > 0$. 
Assume~\ref{eq:LossAssump},~\ref{eq:SubGaussianAssump}. 
With $T$ defined in~\eqref{eq:cone-lasso-penalized},
assume that $\exists\theta>0$ s.t.
for all $u\in T$ with $\|u\|_K \le 1$,
\begin{equation}\label{eq:RSC}\textstyle
\theta^2\|u\|_{K}^2 \le \frac{1}{n}\sum_{i=1}^n \left\{\ell(Y_i, X_i^{\top}\beta^* + X_i^{\top}u) - \ell(Y_i, X_i^{\top}\beta^*) - u^{\top}X_i \ell'(Y_i, X_i^{\top}\beta^*)\right\},
\end{equation}
as well as
\begin{equation}\label{eq:AssumptionSP}
L(2 + 5\xi)\sqrt{2s\log(p/s)/n} \le B_3^{1/2}\phi(T)\theta^2\times\begin{cases}1/\sigma^*,&\mbox{for $\ell$, the squared loss},\\2,&\mbox{for $\ell$, the logistic loss.}\end{cases}
\end{equation}
Then with probability at least $1 - 2/(\xi^2\log(p/s)(p/s)^{\xi})$,
\begin{equation}\label{eq:EstimationErrorRates}
\|\hat{\beta} - \beta^*\|_K \le \frac{L(2 + 5\xi)}{B_3^{1/2}\phi(T)\theta^2}\sqrt{\frac{2s\log(p/s)}{n}}\times\begin{cases}\sigma^*,&\mbox{for $\ell$, the squared loss},\\
0.5,&\mbox{for $\ell$, the logistic loss.}\end{cases}
\end{equation}
\end{restatable}
The proof is given \Cref{sec:RSCLogistic}.
Assumption~\eqref{eq:RSC} is the classical restricted strong convexity condition 
and 
we verify this for linear and logistic loss in~\Cref{prop:RSCVerifiedC}.
Results similar to~\Cref{prop:RateLogisticLasso-mainpaper} are known in the literature~\citep{negahban2009unified} but the main novelty of our result is that the tuning parameter $\lambda$ is of order $\sqrt{\log(p/s)/n}$ and not $\sqrt{\log(p)/n}$ which proves the minimax optimal rate.
\paragraph{Faster rates for the penalized Lasso.}

Fast rates for the Lasso can be obtained using the second inequality of
\Cref{thm:main}, which when specialized to the squared loss gives
\eqref{eq:square-loss-fast-rate}.
To verify the main additional assumption of $\hat{\beta}-\eta\in T$, we prove sparsity of $\eta$ and $\hat{\beta}$.
Since $\hat{\beta},\eta$ are defined through a penalized quadratic problem, we can leverage existing results in the literature that imply that
$\eta,\hat{\beta}$ satisfies $\|\eta\|_0\vee\|\hat{\beta}\|_0 \le \tilde C s$
under suitable conditions on the design and as long as $s\log(ep/s)/n$ is small enough,
for some constant $\tilde C$ that depends
on the restricted singular values of $\Sigma$; cf.,
e.g.,\cite[Lemma 1]{zhang2010nearly}, \cite[Theorem 3]{belloni2014pivotal} \cite[Lemma 3.5]{javanmard2015biasing},
\cite[Section 7.1]{bellec_zhang2019debiasing_adjust}.
We prove such as result for the Group-Lasso in \Cref{prop:SparsityEta} below.
Now we define the cones $T_0$ and $T$ as the sets
\begin{equation}
    \label{cone-T_0-lasso-sparse}
    T_0 \coloneqq \{u\in\R^p: \|u\|_0\le (2\tilde{C} + 1) s\}
    \subset T = \{u\in\R^p: \|u\|_1 \le (2\tilde{C} + 1)^{1/2} \sqrt s\|u\|\}.
\end{equation}
where the inclusion is obtained thanks to the Cauchy-Schwarz inequality.
Then $\{\eta-\hbeta,\hbeta-\beta^*,\eta-\beta^*\}\subset T$ with high probability,
the Gaussian width $\gamma(T,\Sigma)$ is bounded by \Cref{lemma:upper-bound-cone-T-lasso}
and the second inequality of \Cref{thm:main} yields 
$$\|\Sigma^{1/2}(\eta-\hbeta)\| \lesssim
L^2r_n\mathcal{E},
\quad \text{where} \quad 
r_n =\phi(T)^{-1} (s \log(ep/s)/n)^{1/2}.$$
Since $\mathcal{E} \lesssim r_n$ with high
probability by known prediction bounds for the Lasso 
(see~\Cref{prop:RateLogisticLasso-mainpaper} and its proof in \Cref{sec:RSCLogistic} for rates with squared and logistic loss),
we obtain that with high probability,
\begin{equation}
    \label{fast-rate-lasso}
    \|\Sigma^{1/2}(\eta-\hbeta)\| \lesssim L^3\phi^{-2}(T) s \log(ep/s)/n = L^3r_n^2,
\end{equation}
a rate that is the square of the minimax rate $r_n$, hence much smaller. 
For squared loss, this rate is also faster than the rate obtained in \eqref{eq:constrained-lasso-upper}
which is of order $r_n^{3/2}$. This faster rate is obtained thanks
to the inclusion $\{\hbeta-\eta,\hbeta-\beta^*,\eta-\beta^*\}\subset T$, whereas in the setting
of \eqref{eq:constrained-lasso-upper} we only had $\{\hbeta-\beta^*,\eta-\beta^*\}\subset T$ but not $\hbeta-\eta\in T$.
To our knowledge, the only result in the literature similar to the above bounds is given by \cite[Theorem 5.1]{javanmard2015biasing}.
This result from \cite{javanmard2015biasing} shows that \eqref{fast-rate-lasso} holds for squared loss,
provided that the covariance $\Sigma$ satisfies (a) the minimal singular value of $\Sigma$ is at least $c_3>0$, (b) the maximal singular value of $\Sigma$ is at most $c_4$,
and (c) the covariance matrix $\Sigma$ satisfies
\begin{equation}
    \textstyle
    \max_{A\subset [p]: |A|\le c_5 s}
    \max_{j\in A}
    \sum_{j\in A^c} |\Sigma_{ij}| \le c_6.
    \label{assum:javanmard}
\end{equation}
Our results show that a first order expansion
for the Lasso can be obtained using the slow rate bound \eqref{eq:square-loss-slow-rate} without the requirement that the spectral norm
of $\Sigma$ is bounded, and for the fast rate without the stringent assumption \eqref{assum:javanmard}
on the correlations of $\Sigma$.
Not only do our results generalize Theorem 5.1 from \cite{javanmard2015biasing}
to more general $\Sigma$, \Cref{thm:main} shows how to obtain first-order
expansion $\eta$ beyond the squared loss (e.g. logistic loss) and beyond
the $\ell_1$-penalty of the lasso: the previous subsection
tackles the constrained Lasso penalty \eqref{eq:penalty-constrained-lasso}
and the next subsection tackles the Group-Lasso penalty.

Sparsity of $\eta$ for any general loss function is proved in~\Cref{prop:SparsityEta}. 
This alone does not imply inclusion of $\eta - \hat{\beta}$ in 
a low-dimensional set without sparsity of $\hat{\beta}$. 
Sparsity of $\hat{\beta}$ for general loss function is not well-studied but for logistic loss function Section D.4 of the supplement of~\cite{belloni2016post} proves a sparsity bound of the form $\|\hbeta\|_0 \le \tilde C s$, similar to the squared loss.
Unfortunately the proof there requires $\lambda \gtrsim \sqrt{\log p/n}$ instead of
condition $\lambda\gtrsim \sqrt{\log(p/s)/n}$ used in
\Cref{lemma:penalizedLassoBelongstoCone} above and
in~\cite{lecue2015regularization_small_ball_I,sun2013sparse,bellec2016slope,bellec_zhang2019debiasing_adjust,bellec2018nb_lsb}.

\subsection{Group-Lasso}
\label{subsec:GL}
Consider now a partition of $\{1,...,p\}$ into $M$ groups $G_1,...,G_M$.
For simplicity, we assume that the groups have the same size $d=p/M$,
which is typically the case in multitask learning with $d$ tasks and $M$ shared
features. The Group-Lasso penalty studied in this subsection is
\begin{equation}
\textstyle
\label{h-GL}
h(\beta) = \lambda \sum_{k=1}^M \|\beta_{G_k}\|
\qquad\text{ where }\beta_{G_k}\in\R^{|G_k|} \text{ is the restriction } (\beta_j, j\in G_k).
\end{equation}
In both the linear model with squared loss and in logistic regression with the logistic loss, we now show that 
$\hbeta-\beta^*$ and $\eta-\beta^*$ belong to a low-dimensional cone
(\Cref{lemma:GroupLassoBelongstoCone}),
and that the Gaussian width of this cone is bounded from above by
$\sqrt{s}(\sqrt d + \sqrt{2\log(M/s)})$ where $s$ is the number of groups
with $\beta^*_{G_k}\ne 0$ (\Cref{lemma:GroupLassoBoundGamma}).

\begin{restatable}{lemma}{lemmaGroupLassoBelongsToCone}
    \label{lemma:GroupLassoBelongstoCone}
    Consider the linear model \eqref{eq:linear-model} and assume
    that $\max_{k=1,...,M}\|\Sigma_{G_k,G_k}\|_{op}\le 1$ 
    and that each group has the same size $|G_k|=d=p/M$.
    Let $\xi>0$ and set $\lambda = L \sigma^*(1+\xi)[\sqrt{d}+(1+2\xi)\sqrt{2\log(M/s)}]$
    where
    $(\sigma^*)^2=(\sum_{i=1}^n\eps_i^2)/n$ and $s$ is the number of groups
    with $\beta^*_{G_k}\ne 0$. 
    Then
    \begin{equation}
        \label{eq:belongs-to-cone-group-lasso-lemma-conclusion2}
        \P\left(
        \{\hbeta-\beta^*,\eta-\beta^*\} \subset
        T
        \right)
        \ge
        1 - 2\big/\left(2\xi^2\log(M/s)(M/s)^\xi\right).
    \end{equation}
    for $T =
        \{\delta\in\R^p:
            \sum_{k=1}^M\|\delta_{G_k}\| \le \sqrt s \|\delta\|_2(2+{3}{\xi^{-1}})
        \}$.
    If instead the logistic regression model and assumptions of
    \Cref{prop:logistic-setting} are fulfilled and 
    $\lambda$ is as above with $\sigma^*=1/2$, then \eqref{eq:belongs-to-cone-group-lasso-lemma-conclusion2} also holds.
\end{restatable}
The fact that the Group-Lasso belongs with high probability to a low-dimensional
cone has been used before to prove risk bounds, e.g., \cite{lounici2011oracle,bellec2017towards}.
However the tuning parameter in the above lemma is smaller than that used in these works
and using such cones to prove first expansion as in the present paper are, to our knowledge, novel.

\begin{restatable}{lemma}{lemmaGroupLassoBoundGamma}
    \label{lemma:GroupLassoBoundGamma}
    Assume that $\max_{k=1,...,M}\|\Sigma_{G_k,G_k}\|_{op}\le 1$ 
    and that each group has the same size $|G_k|=d=p/M$.
    The set $T$ defined in the previous lemma satisfies
    $\gamma(T,\Sigma) \lesssim C(\xi) \phi(T)^{-1} \sqrt{sd + s\log(M/s)}$
    for some constant $C(\xi)$ that depends only on $\xi$.
\end{restatable}

Hence if the number of groups $M$,
the group-sparsity $s$ (number of groups such that $\beta^*_{G_k}\ne 0$) and the group size $d=p/M$ satisfy
$(sd + s\log(M/s))/n \to 0$
while $\phi(T)$ is bounded away from 0,
the above Lemmas combined with \Cref{thm:main} imply that $\eta$ is a first-order 
expansion of $\hbeta$ 
for both the squared loss in linear regression
and logistic loss in the logistic model.
We leverage this result to obtain an exact risk identity for the Group-Lasso in
the next section.

\begin{restatable}{proposition}{propSparsityEta}
    \label{prop:SparsityEta}
    Assume~\ref{eq:LossAssump},~\ref{eq:SubGaussianAssump}.
    Let the setting of \Cref{lemma:GroupLassoBoundGamma} be fulfilled.
    Fix $\lambda$ as in \Cref{lemma:GroupLassoBelongstoCone}
for both squared and logistic loss for some $\xi > 0$ and $T$ be the cone defined in~\Cref{lemma:GroupLassoBelongstoCone}. If $\|K\|_{op} \le C_{\max} < \infty$ and the assumptions of~\Cref{prop:RateLogisticLasso-mainpaper} hold, then 
\[
\mathbb{P}\left( |\{k\in[M]:\eta_{G_k} \ne 0 \}|
    \le s \tilde C \right) \ge
    1 - 2/(\xi^2\log(M/s)(M/s)^\xi),
\]
where $\tilde{C} := 1 + C_{\max} \{2(3+\xi)(1+\xi^{-1}) \}^2B_3^2\phi(T)^{-2}$.
For the squared loss, the same holds for $\hat{\beta}$
with $\tilde{C}$ replaced by $(1 + o(1))\tilde{C}$
provided $\phi(T)^{-1}\sqrt{sd+s\log(M/s)}/\sqrt n \to 0$.
\end{restatable}
The proof is given in \Cref{section-sparsity-eta-appendix}.
For the Lasso the assumption of $\|K\|_{op} \le C_{\max}$ can be relaxed to a bound
on the sparse maximal eigenvalue of $K$ using devices from
\cite[Lemma 1]{zhang2010nearly}, \cite[Corollary 2]{zhang2012},
\cite[Lemma 3]{belloni2013least}
or \cite[Proposition 7.4]{bellec_zhang2019debiasing_adjust}. See also
\cite[Theorem 3.1]{lounici2011oracle} and \cite[Lemma 6]{liu2009estimation}
for similar results for the Group-Lasso, although with a larger tuning parameter
than in \Cref{prop:SparsityEta}.

For the squared loss, if the condition number of $\Sigma$ stays bounded
then $C_{\max}/\phi(T)^{-2}$ is also bounded.
Then if $r_n = \sqrt{sd +s\log(M/s)}/\sqrt n\to 0$,
\Cref{prop:SparsityEta} yields that both $\hbeta-\eta$
belongs to the cone $\{\delta\in\R^p: \sum_{k=1}^M \|\delta_{G_k}\|
\le (1+o(1))(2\tilde C s)^{1/2} \|\delta\|_2\}$,
which yields the ``fast rate'' bound \eqref{eq:square-loss-fast-rate}.

For the square loss, \Cref{table:squared} in the supplement
summarizes the tuning parameters,
minimax rates for $\|\hat\beta-\beta^*\|_\Sigma$, Gaussian width bound and
upper bounds  on $\|\hat\beta-\eta\|_\Sigma$
for the Lasso and Group-Lasso.
\Cref{table:logistic} gives an analogous
summary for the logistic loss.

\section{Application to exact risk identities}\label{sec:ExactRiskBound}

In the linear model with the squared loss and identity covariance $(\Sigma=I_p)$,
the expansion $\eta$ in \eqref{eq:FirstOrderApproxRegularized}
is particularly simple: $\eta$ becomes the proximal operator
of the penalty $h$ at the point $z= \beta^* + n^{-1/2}\sum_{i=1}^n \eps_i X_i$,
i.e, $\eta=\prox_h(z)$ where $\prox_h(x)=\argmin_{b\in\R^p}\|x-b\|^2/2+h(b)$.
Hence the loss $\|\eta-\beta^*\|$ of $\eta$ has a simple form and
if a first-order expansion \eqref{approximation-intro} is available, for
instance for the Lasso or Group-Lasso as a consequence of the Lemmas of the
previous section,
then the loss $\|\hbeta-\beta^*\|$ is exactly the loss of $\prox(z)$ 
up to a smaller order term.
Let us emphasize that the next result and following discussion provide
exact risk identities for the loss $\|\hbeta-\beta^*\|$ (as in
\eqref{eq:ExactRiskSimplified} below), and not only upper bounds up to
multiplicative constants.

\begin{restatable}{theorem}{theoremExactRiskIdentity}[Exact Risk Identity]\label{thm:ExactRiskBound}
Consider the linear model \eqref{eq:linear-model} and the regularized problem~\eqref{hbeta-intro} with an arbitrary proper convex function $h(\cdot)$. Assume that $X_1, \ldots, X_n$ are iid $N(0, I_p)$ independent of $\eps_1,...,\eps_n$
and set $\sigma^*=(\frac 1 n \sum_{i=1}^n \eps_i^2)^{1/2}$.
Then with probability at least $1-2\exp(-t^2/2)$,
\begin{equation}\label{eq:ExactRiskBound}
\Big| \|\hat{\beta} - \beta^*\| - \mathbb{E}_Z\left[\|\beta^*-\prox_h(\beta^* + n^{-1/2}\sigma^*Z)
\|^2\right]^{1/2}
 \Big|
\le 
\frac{\sigma^*(t+1)}{n^{1/2}} + \|\hbeta - \eta \|
\end{equation}
where $Z=\tfrac{1}{n^{1/2}\sigma^*} \sum_{i=1}^n\eps_i X_i\sim N(0, I_p)$ and $\mathbb{E}_Z$ denotes the expectation with respect $Z$.
\end{restatable}
Theorem~\ref{thm:ExactRiskBound} is a generalization of Corollary 5.2 of~\cite{javanmard2015biasing} where the result is stated for $h(\beta) = \lambda\|\beta\|_1$
with $\lambda \gtrsim \sigma^*\sqrt{2\log(p)/n}$.
For the case of Lasso, either in its constrained
form with tuning parameter chosen as in \Cref{lemma:penalizedLassoBelongstoCone}
or the penalized Lasso with tuning parameter as in \Cref{lemma:penalizedLassoBelongstoCone}, inequality~\eqref{eq:constrained-lasso-upper} holds
thanks to \eqref{set-T-constrained-lasso} and \Cref{lemma:upper-bound-cone-T-lasso}
for the constrained Lasso, and thanks to \Cref{lemma:upper-bound-cone-T-lasso,lemma:penalizedLassoBelongstoCone} for the penalized Lasso.
Hence for both the constrained and penalized Lasso, if $\Sigma = I_p$
with Gaussian design,
the second term on the right hand side of~\eqref{eq:ExactRiskBound} is
$O_p(\sigma^*/\sqrt n) + O_p(s\log(ep/s)/n)^{1/4})(\|\eta-\beta^*\|+\|\hbeta-\beta^*\|)$.
Hence if $s,n,p\to+\infty$ with $s\log(ep/s)/n\to 0$ and $s/p\to0$ then \eqref{eq:ExactRiskBound}
implies
\begin{equation}\label{eq:ExactRiskSimplified}
\|\hat{\beta} - \beta^*\| = (1 + o_p(1))\mathbb{E}_Z[\|\beta^*-\prox_h(\beta^* + n^{-1/2}\sigma^*Z)
\|^2]^{1/2}.
\end{equation}
For the penalized Lasso, since $\eta$ represents a soft-thresholding operator which can be written in
closed form, Theorem~\ref{thm:ExactRiskBound} allows a refined study of the
risk of $\hat{\beta}$; see~\cite[Theorem 5.1]{candes2006modern}.
Similarly for the Group-Lasso, we have from
Lemmas~\ref{lemma:GroupLassoBelongstoCone} and~\ref{lemma:GroupLassoBoundGamma}
that $\|{\eta} - \hbeta\| = O_p((sd +
s\log(M/s))^{1/4}/n^{1/4})\|\hat{\beta} - \beta^*\|$ (slow rate) which is again
negligible relative to $\|\hbeta - \beta^*\|$ if $(sd + s\log(M/s))/n\to0, s/M\to0$.
Thus,~\eqref{eq:ExactRiskSimplified} again holds true. For the Group-Lasso
$\eta=\prox_h(\beta^* + n^{-1/2}\sigma^*Z)$ represents the Block James-Stein estimator
in the sequence model; see~\cite[Section 2.1]{cai2009data}.

Extending Corollary 5.2 of \cite{javanmard2015biasing} to more general loss/penalty functions,
the above device lets us characterize
the risk $\|\hat\beta-\beta^*\|$:
Up to a multiplicative constant of order $1+o_p(1)$, the risk is the same as 
the risk of the proximal of $h$ in the Gaussian sequence model where one observes
$N(\beta^*, (\sigma^*)^2/n)$.


\section{Application to inference}\label{sec:Inference}

The second application we wish to mention is related to confidence intervals
in the linear model when the squared loss is used
and $X_1,...,X_n$ are iid Gaussian $N(0,\Sigma)$.
Assume that one is interested in constructing a confidence interval
for a specific linear combination $a^T\beta^*$ for some $a\in\R^p$.
Further assume, for simplicity, that $\Sigma$ is known and that $a$
is normalized with $\|\Sigma^{-1/2}a\|=1$. Then previous works
on \emph{de-biasing} \cite{zhang2011statistical,ZhangSteph14,JavanmardM14a,JavanmardM14b,GeerBR14,javanmard2015biasing,bellec_zhang2019debiasing_adjust} suggests, given an estimator 
$\hbeta$ that may be biased, to consider the bias-corrected estimator
$\hat \theta$ defined by
$\hat \theta = a^T\hbeta + \|z_a\|^{-2}z_a^T(y-\X\hbeta)$,
where $y=(Y_1,...,Y_n)$ is the response vector and $\X$ is the design
matrix with rows $X_1,...,X_n$ and $z_a= \X \Sigma^{-1} a \sim N(0,I_n)$ is sometimes
referred to as a score vector for the estimation of $a^T\beta^*$.

\begin{restatable}{proposition}{propositionDebiasing}
    \label{prop:debiasing}
    Assume that $X_1,...,X_n$ are iid $N(0,\Sigma)$ and is independent
    of $\eps=(\eps_1,...,\eps_n)\sim N(0,I_n)$.
    Assume that for some cone $T$ and $r_n = \gamma(T,\Sigma)/\sqrt n$ we have
    \begin{equation}
        \label{assum:alpha-debiasing}
    \P(\|\Sigma^{1/2}(\hbeta-\beta^*)\|+\|\Sigma^{1/2}(\eta-\beta^*)\|\le \C r_n,
\{\eta-\hbeta,\eta-\beta^*,\hbeta-\beta^*\} \subset T)\ge 1-\alpha.
    \end{equation}
    Then for some
    $T_n$ with the $t$-distribution with $n$ degrees-of-freedom,
    with probability $1-\alpha- 4 e^{-nr_n^2/2}$,
    \begin{align}
        \label{de-biasing-conclusion}
        \sqrt n (\hat \theta - a^T\beta^*)
        - T_n
        &= O_p((1+r_n)) \|\Sigma^{1/2}(\eta-\beta^*)\| + O_p(\sqrt n r_n) \|\Sigma^{1/2}(\eta-\hbeta)\|,
        \\&= O_p(r_n(1+r_n)) + O_p(\sqrt n r_n^3).
        \label{de-biasing-conclusion-2}
    \end{align}
\end{restatable}

Because $T_n$ has $t$ distribution with $n$ degrees of freedom, asymptotically $\mathbb{P}(|T_n| \le 1.96) \to 0.95$ and hence from~\eqref{de-biasing-conclusion-2}, we get that $\mathbb{P}(n^{1/2}|\hat{\theta} - a^{\top}\beta^*| \le 1.96) \to 0.95$ if $r_n^3\sqrt n\to0$. Therefore, $[\hat{\theta}-1.96/n^{1/2}, \hat{\theta}+1.96/n^{1/2}]$ represents a $95\%$ confidence interval for $a^{\top}\beta^*.$
Conclusion \eqref{de-biasing-conclusion} is a consequence of \Cref{thm:main}.

\vspace{-0.12in}

\subparagraph{Lasso.} Eq. \eqref{assum:alpha-debiasing} is satisfied for the penalized Lasso for $r_n=\sqrt{s\log(ep/s)/n}$
and the cone $T$ in \eqref{cone-T_0-lasso-sparse},
in situations where $\|\hbeta\|_0\le \tilde C s$ with high probability
as explained in the discussion surrounding \eqref{cone-T_0-lasso-sparse}.
In order to construct confidence interval based on \eqref{de-biasing-conclusion},
the right hand side of \eqref{de-biasing-conclusion-2} needs to converges to 0.
This is the case if $r_n\to 0$ and $\sqrt n r_n^3\to0$.
For the Lasso with $r_n=s\log(ep/s)/n$, this translates to the sparsity
condition $s^3\log(ep/s)^3/n^2 \to 0$, i.e., $s = o(n^{2/3})$ up to logarithmic factors.
Hence the first order expansion results of the present paper lets us
derive de-biasing results for the Lasso beyond the condition $s\lesssim \sqrt n$
required in the early results
\cite{ZhangSteph14,JavanmardM14a,GeerBR14}
on de-biasing (other recent approaches, \cite{javanmard2015biasing,bellec_zhang2019debiasing_adjust} also allow to prove such result beyond $s\lesssim \sqrt n$). Moreover, the above proposition is general and apply to any regularized
estimator such that \eqref{assum:alpha-debiasing} holds, with suitable bounds on
the Gaussian complexity $\gamma(T,\Sigma)$.
For $s\ggg n^{2/3}$, the estimator $\hat\theta$
requires an adjustment for asymptotic normality
in the form a degree-of-freedom adjustment \cite{bellec_zhang2019debiasing_adjust}.

\vspace{-0.12in}

\subparagraph{Group-Lasso.}
If $s$ is the number of non-zero groups,
$r_n=\sqrt{sd+s\log(M/s)}/\sqrt n$ and the condition number of $\Sigma$ is bounded,
then \eqref{assum:alpha-debiasing}
holds thanks to \Cref{prop:SparsityEta}, the last paragraph of
\Cref{subsec:GL} and the risk bound \eqref{risk-bound-eta-GL}.
Here, \eqref{de-biasing-conclusion-2} is $o(1)$ if and only if
$(sd + s\log(M/s))/n^{2/3}\to 0$. This improves the sample size requirement
of \cite{mitra2016benefit}, although $\Sigma$ is assumed known
in \Cref{prop:debiasing}.

\section{Proof sketch of \Cref{thm:main}
    {\quad \small (Detailed proofs are given in \Cref{appendix:proof-sketch-details})}
}
\label{sec:proof-sketch}

\begin{restatable}{theorem}{theoremSketchDeterministic}\label{thm:Deterministic}
Define $\hat{K} := n^{-1}\sum_{i=1}^n \ell''(Y_i,X_i^{\top}\beta^*)X_iX_i^{\top}$.
Under assumption~\ref{eq:LossAssump}, we have

(i) If $\{\hat{\beta}-\beta^*, \eta - \beta^*\}\subseteq T$ then $\|\hat{\beta} - \eta\|_K \lesssim Q_{n,1}^{1/2}\mathcal{E} + B^{1/2}Z_n^{1/2}\mathcal{E}^{3/2}$.

(ii) If $\{\hbeta -\eta, \hbeta-\beta^*, \eta - \beta^*\}\subseteq T$ then $\|\hat{\beta} - \eta\|_K \lesssim Q_{n,2}\mathcal{E} + BZ_n\mathcal{E}^2$,

where
\[
Q_{n,1} = \sup_{u\in T}\,\left|\frac{u^{\top}\hat{K}u}{\|u\|_K^2} - 1\right|,\; Q_{n,2} = \sup_{u,v\in T}\,\frac{|u^{\top}(\hat{K} - K)v|}{\|u\|_K\|v\|_K}\;\;\mbox{and}\;\; Z_n = \sup_{u\in T}\frac{1}{n}\sum_{i=1}^n \frac{|X_i^{\top}u|^3}{\|u\|_K^3}.
\]
\end{restatable}
\Cref{thm:Deterministic} follows from the strong convexity of the objective
function of $\eta$ with respect to the norm $\|\cdot\|_K$
(cf. for instance, Lemma 1 of~\cite{bellec2018prediction}) combined with
Taylor expansions of the loss $\ell$.
Next, to prove Theorem~\ref{thm:main}, it remains to bound $Q_{n,1}(T), Q_{n,2}(T)$ and $Z_n(T)$. The quadratic processes $Q_{n,1}(T), Q_{n,2}(T)$ and cubic process $Z_n(T)$ can be bounded in terms of $\gamma(T, \Sigma)$ using
generic chaining results,
Theorem 1.13 of~\citet{mendelson2016upper} and Eq. (3.9) of~\cite{mendelson2010empirical},
as follows.
\begin{restatable}{proposition}{propositionControlProc}[Control of $Q_{n,1}, Q_{n,2}$ and $Z_n$]\label{prop:ControlProc}
Under assumptions~\ref{eq:LossAssump} and~\ref{eq:SubGaussianAssump}, we have 

(i) With probability $1 - 2\exp(-\C t^2\gamma^2(T,\Sigma))$,
\[
    \textstyle
\max\{Q_{n,1}(T), Q_{n,2}(T)\} \le \C B_2B_3L^2\left(tn^{-1/2}\gamma(T,\Sigma) + t^2n^{-1}\gamma^2(T,\Sigma)\right).
\]
(ii) With probability $1 - 2\exp(-\C t\log n)$, \;\;
$
Z_n(T) \le \C B_3^{3/2}L^3\left(1 + n^{-1}\gamma^3(T,\Sigma)\right)t^3.
$
\end{restatable}


\newpage

\bibliographystyle{plainnat}
\bibliography{biblio}

\newpage
\appendix

\section*{SUPPLEMENT}

\begin{table}[h]
\begin{tabular}{@{}|l|l|l|@{}}
\toprule
                         & Lasso  & Group-Lasso, $M$ groups of size $d=p/M$  \\ \midrule
    Tuning parameter & $\lambda \gtrsim [\frac{2}{n}\log\frac{p}{s}]^{\frac 1 2}$ in \eqref{h-penalized-lasso} & $\lambda \gtrsim[d+\frac{2}{n}\log\frac{M}{s}]^{\frac 1 2}$ in \eqref{h-GL} \\ \midrule
                     & $\|\hat\beta-\beta^*\|_{\Sigma}\lesssim r_n$  & $\|\hat\beta-\beta^*\|_{\Sigma}\lesssim r_n$
    \\
    Minimax rate $r_n$             & $r_n =[\frac{2s}{n}\log\frac{p}{s}]^{\frac 1 2}$  & $[\frac{d}{n}+\frac{s}{n}\log\frac{M}{s}]^{\frac 1 2 }$
 \\ \midrule
    Gaussian width bound             &
    $\gamma(T, \Sigma)\lesssim [s\log\frac p s]^{1/2}$ 
    & 
    $\gamma(T, \Sigma)\lesssim [sd + s\log\frac M s]^{1/2}$ 
    by \Cref{lemma:GroupLassoBoundGamma}
 \\ \midrule
    Restricted Eigenvalue (RE)              & \multicolumn{2}{l|}{
        $\|\eta-\hat\beta\|_{\Sigma}\lesssim r_n^{3/2}$

        by \eqref{eq:square-loss-slow-rate} and
        \Cref{lemma:penalizedLassoBelongstoCone} (Lasso) or
        \Cref{lemma:GroupLassoBelongstoCone} (GL)

    } \\ \midrule
    $\|\Sigma\|_{op}\vee \|\Sigma^{-1}\|_{op}\le C$                 & \multicolumn{2}{l|}{
        $\|\eta-\hat\beta\|_{\Sigma}\lesssim r_n^{2}$

        by \eqref{eq:square-loss-fast-rate},
        \Cref{prop:SparsityEta}
    } \\ \bottomrule
\end{tabular}
\caption{Summary of rates for $\|\hbeta-\beta^*\|_\Sigma$ and $\|\eta-\hbeta\|_\Sigma$ for
the squared loss. For the Lasso, $s$ is the sparsity of $\beta^*$ while for
the Group-Lasso $s$ is the number of non-zero groups in $\beta^*$.
If $r_n\to 0$ then $\|\eta-\hbeta\|_\Sigma$ is an order of magnitude smaller than
$\|\hbeta-\beta^*\|_\Sigma$ and the minimax rate.
In this table $\gtrsim$ may hide constants depending on the subgaussian parameter $L$
as well as restricted eigenvalues of $\Sigma$, denoted by $\phi(T)$ in the
paper.
\label{table:squared}
}
\end{table}

\begin{table}[!h]
\begin{tabular}{@{}|l|l|l|@{}}
\toprule
                         & Lasso  & Group-Lasso, $M$ groups of size $d=\frac{p}{M}$  \\ \midrule
    Tuning parameter & $\lambda \gtrsim [\frac{2}{n}\log\frac{p}{s}]^{\frac 1 2}$ in \eqref{h-penalized-lasso} & $\lambda \gtrsim[d+\frac{2}{n}\log\frac{M}{s}]^{\frac 1 2}$ in \eqref{h-GL} \\ \midrule
             & $\|\hat\beta-\beta^*\|_K\lesssim r_n$  & $\|\hat\beta-\beta^*\|_K\lesssim r_n$
    \\
    Minimax rate $r_n$             & $r_n =[\frac{2s}{n}\log\frac{p}{s}]^{\frac 1 2}$
    (Prop. \ref{prop:RateLogisticLasso-mainpaper})
                                   & $[\frac{d}{n}+\frac{s}{n}\log\frac{M}{s}]^{\frac 1 2 }$
 \\ \midrule
    Gaussian width bound             &
    $\gamma(T, \Sigma)\lesssim [s\log\frac p s]^{1/2}$ 
    & 
    $\gamma(T, \Sigma)\lesssim [sd + s\log\frac M s]^{1/2}$ 
    by \Cref{lemma:GroupLassoBoundGamma}
\\ \midrule
    $RSC$ (cf. \Cref{sec:RSCLogistic})                 & \multicolumn{2}{l|}{
        $\|\eta-\hat\beta\|_K\lesssim r_n^{3/2}(1+ r_n^3\sqrt n)$

        by \eqref{eq:slow-rate-logistic} and
        \Cref{lemma:penalizedLassoBelongstoCone} or
        \Cref{lemma:GroupLassoBelongstoCone}

    } \\ \midrule
    $\|K\|_{op}\vee \|K^{-1}\|_{op}\le C$             & \multicolumn{2}{l|}{
        $\|\eta-\hat\beta\|_K\lesssim r_n^{2}(1+ r_n^3\sqrt n)$

        by \eqref{eq:fast-rate-logistic} and
        \Cref{prop:SparsityEta}
    } \\ \bottomrule
\end{tabular}
\caption{Summary of rates for $\|\hbeta-\beta^*\|_K$ and $\|\eta-\hbeta\|_K$ for
the logistic loss. For the Lasso, $s$ is the sparsity of $\beta^*$ while for
the Group-Lasso $s$ is the number of non-zero groups in $\beta^*$.
If $r_n\to 0$ as well as $r_n^3\sqrt n\to 0$
then $\|\eta-\hbeta\|_K$ is an order of magnitude smaller than
$\|\hbeta-\beta^*\|_K$ and the minimax rate.
In this table $\gtrsim$ may hide constants depending on the subgaussian parameter $L$,
the constants $B_3$
and Restricted Strong Convexity (RSC) constants.
\label{table:logistic}
}
\end{table}

\paragraph{Proofs.}
All Theorems, Lemmas and Propositions from the submission are proved in the
present supplement.  The results are restated before their proofs for
convenience.

\section{Proofs of \Cref{sec:4-examples-of-cones-T}}
\lemmaUpperBoundConeTLasso*
\begin{proof}[Proof of \Cref{lemma:upper-bound-cone-T-lasso}]
    This is a consequence of \Cref{lemma:GroupLassoBoundGamma} proved below,
    by taking $p$ groups of size $d=1$, i.e., the groups are
    $G_j=\{j \}$ for each $j=1,...,p$ and $M=p$.
    The condition $\max_{k=1,...,M}\|\Sigma_{G_k,G_k}\|_{op}\le 1$
    necessary to apply \Cref{lemma:GroupLassoBelongstoCone} is equivalent
    to the normalization \ref{eq:Normalized}.
\end{proof}

\lemmaRiskConstrainedLasso*
\begin{proof}[Proof of \Cref{lemma:risk-constrained-lasso}]
    Since $R=\|\beta^*\|_1$, the inclusion \eqref{set-T-constrained-lasso}
    holds by the triangle inequality, i.e., we have $\hbeta-\beta^*\in T$
    as well as $\eta-\beta^*\in T$.

    Next, we first bound the loss of $\eta$.
    The optimization problem \eqref{eq:FirstOrderApproxRegularized}
    for the squared loss for the penalty \eqref{eq:penalty-constrained-lasso}
    can be rewritten as
    $$\eta = \argmin_{\beta\in\R^p: \|\beta\|_1\le R} \frac 1 2 \|\Sigma^{1/2}(\beta-\beta^*) - n^{-1/2}Z \|^2, 
    \qquad \text{ where }\qquad 
    Z = \frac{1}{\sqrt n}\sum_{i=1}^n \eps_i \Sigma^{-1/2} X_i 
    .
    $$
    By optimality of $\eta$ for the above optimization problem,
    we have (see, e.g., the properties of convex projections in \cite{bellec2018sharp})
    that
    $$\|\Sigma^{1/2}(\eta-\beta^*)\| 
    \le \frac{1}{\sqrt n} \frac{(\Sigma^{1/2}(\eta-\beta^*))^T Z}{\|\Sigma^{1/2}(\eta-\beta^*)\|} 
    \le \frac{\sigma^*}{\sqrt n} \sup_{u\in T:\|\Sigma^{1/2}u\|=1} u^T\Sigma^{1/2} Z/\sigma^*.$$

    Next, notice that $Z/\sigma^*$ is $L$-subgaussian because for any $u\in\R^p$,
    by independence,
    \begin{equation}
        \label{eq:Z-is-subgaussian-squared-loss}
    \E[\exp(u^TZ)]
    = \prod_{i=1}^n \E[e^{n^{-1/2}\eps_i X_i^T\Sigma^{-1/2}u}]
    \le \prod_{i=1}^n e^{n^{-1}\eps_i^2 L^2 \|u\|^2 / 2} 
    = e^{(\sigma^*)^2 L^2 \|u\|^2/2}.
    \end{equation}
    By a tail bound on suprema of subGaussian processes, we obtain
    that with probability at least $1-e^{-t^2}$, inequality
    $\sup_{u\in T:\|\Sigma^{1/2}u\|=1} u^T\Sigma^{1/2} Z/\sigma^*
    \le \Cl{abs-proof}(
    \gamma(T,\Sigma) + t)
    $ holds for some absolute constant $\Cr{abs-proof}$.
    We have proved in \Cref{lemma:upper-bound-cone-T-lasso} that
    $\gamma(T,\Sigma) \le \C \phi(T)^{-1} \sqrt{s\log(2p/s)}$
    for another absolute constant.
    The choice $t=\phi(T)^{-1} \sqrt{s\log(2p/s)} = r_n\sqrt n$ completes the proof
    for $\eta$.

    We now prove the bound for $\hbeta$. For the squared loss in the linear model,
    $$\hbeta = \argmin_{\beta\in\R^p: \|\beta\|_1\le R} \|\X(\beta-\beta^*) -
    \eps\|^2/(2n)$$
    where $\X$ is the design matrix with rows $X_1,...,X_n$
    and $\eps=(\eps_1,...,\eps_n)$.
    The optimality conditions of the above optimization problem yields that
    $$\frac{\|\X(\hbeta-\beta^*)\|^2}{n \|\Sigma^{1/2}(\hbeta-\beta^*)\|}
    \le \frac{ \eps^T \X (\hbeta-\beta^*)}{n \|\Sigma^{1/2}(\hbeta-\beta^*)\|}
    = \frac{1}{\sqrt n}
    \frac{(\Sigma^{1/2}(\hbeta-\beta^*))^T  Z}{\|\Sigma^{1/2}(\hbeta-\beta^*)\|} 
\le \frac{\sigma^*}{\sqrt n} \sup_{u\in T:\|\Sigma^{1/2}u\|=1}\frac{u^T\Sigma^{1/2} Z}{\sigma^*}.$$
    We have already bounded in the previous paragraph
    the supremum in the right hand side with probability at least $1-e^{-nr_n^2}$.
    It remains to show that the left hand side is larger than
    $\|\Sigma^{1/2}(\hbeta-\beta^*)\|(1-\C r_n)$ with high probability.
    Since $\hbeta-\beta^*\in T$, an application of \cite{plan_vershynin_liaw2017simple}
    to the set $(\Sigma^{1/2}T) \cap \{v\in\R^p : \|v\|=1\}$
    yields that, with probability at least $1-2e^{-r_n^2n}$,
    $$\Big| \frac{\|\X(\hbeta-\beta^*)\|}{\|\Sigma^{1/2}(\hbeta-\beta^*)\|} - \sqrt n  \Big|
    \le \C( \gamma(T,\Sigma) + \sqrt n r_n )
    \le \C \sqrt n r_n.
    $$
    In the same event, we have
    $\|\X(\hbeta-\beta^*)\|^2/n \ge (1-\C r_n)^2 \|\Sigma^{1/2}(\hbeta-\beta^*)\|^2$
    and the proof is complete.
\end{proof}

The following Lemma will be useful.

\begin{lemma}\label{lem:LemmaA1}
The following, (i) in the linear model and (ii) in the logistic model, hold
for any convex penalty $h$.

    (i)
    Consider the linear model \eqref{eq:linear-model}, assume~\ref{eq:SubGaussianAssump} and assume that $(\eps_1,...,\eps_n)$ is independent of $(X_1,...,X_n)$
            and let $(\sigma^*)^2 = (1/n)\sum_{i=1}^n\eps_i^2$.
            Then almost surely,
            \begin{equation}
                \{\eta,\hbeta\} \subset \hat T = \{b\in\R^p: \sqrt{n}(h(b)-h(\beta^*))\le Z^T\Sigma^{1/2}(b-\beta^*)\}
                \label{set-hat-T-linear-model-squared-loss}
            \end{equation}
            where $Z$ is an $L(\sigma^*)$-subgaussian vector in the sense
            that $\E\exp(u^TZ) \le \exp(L^2(\sigma^*)^2 \|u\|^2/2)$.

    (ii)
    Consider the logistic model and assume~\ref{eq:SubGaussianAssump}.
            Then almost surely
            $$\{\eta,\hbeta\} \subset \tilde T 
            = \{b\in\R^p: \sqrt n ( h(b)-h(\beta^*) )\le \tilde Z^T\Sigma^{1/2}(b-\beta^*)\}$$
            where $\tilde Z$ is an $L/2$-subgaussian vector in the sense that
            $\E\exp(u^T\tilde Z) \le \exp((L/2)^2 \|u\|^2/2)$.
\end{lemma}

\begin{proof}
    (i) We first prove the result in the linear model for the squared loss.
    Here
    \eqref{optimality-additive-penalty} holds
    with $g_n$ and $f_n$ defined before \eqref{optimality-additive-penalty},
    so that
    \begin{equation}
    \sqrt n \nabla f_n(\beta^*) = \sqrt n \nabla g_n(\beta^*)
    = \Sigma^{1/2} Z \qquad \text{ where } \qquad 
    Z \triangleq \frac{1}{\sqrt n} \sum_{i=1}^n \eps_i \Sigma^{-1/2}X_i
    \label{Z-squared-loss}
    \end{equation}
    in the linear model for the squared loss.
    Let $\hat T = \{ b\in\R^p: h(b)-h(\beta^*) \le Z^T\Sigma^{1/2}(b-\beta^*)\}$.
    Then both $\eta$ and $\hbeta$ belong to $\hat T$ by \eqref{optimality-additive-penalty}.
    We already proved that $Z$ is is $L\sigma^*$-subgaussian in the sense that
    $\E[\exp(Z^T u)] \le \exp(L^2(\sigma^*)^2\|u\|^2/2)$ for all $u\in\R^p$
    in \eqref{eq:Z-is-subgaussian-squared-loss}; this completes the proof of (i).

    (ii)
    In logistic regression with the logistic loss,
    \eqref{optimality-additive-penalty} again,
    holds, i.e., $\{\eta,\hbeta\}$ belong to $\tilde T = \{ b\in\R^p: h(b)-h(\beta^*))\le\tilde Z^T\Sigma^{1/2}(b-\beta^*)\}$
    where
    \begin{equation}
    \tilde Z
    \triangleq 
    \sqrt n \Sigma^{-1/2} \nabla f_n(\beta^*) 
    = \sqrt n \Sigma^{-1/2} \nabla g_n(\beta^*)
    = \frac{1}{\sqrt n} \sum_{i=1}^n \left(Y_i - \frac{1}{1+e^{X_i^T\beta^*}}\right) \Sigma^{-1/2}X_i
    \label{Z-tilde-logistic}
    \end{equation}
    where $g_n$ and $f_n$ are defined before \eqref{optimality-additive-penalty}.
    We now show that $\tilde Z\in\R^p$ is a subgaussian vector.
Note that $\E[Y_i|X_i] =  1/(1+e^{X_i^T\beta^*} )$ so that 
    $\E[\tilde Z|X_1,...,X_n]=0$ and $\E[\tilde Z]=0$.
    If $B$ is Bernoulli with parameter $p$,
    then $\E[e^{t(B-p)}] = p e^{t(1-p)} + (1-p) e^{t(-p)}$
    which is maximized at $p=1/2$, hence 
    $\E[e^{t(B-p)}] \le (e^{t/2}+e^{-t/2})/2$.
    For any $u\in\R^p$, set $v=\Sigma^{-1/2}u$ and notice
    that by independence and the law of total expectation,
    $$\E[e^{\tilde Z^T u}]=
    \prod_{i=1}^n \E[e^{(Y_i - \E[Y_i|X_i])\frac{X_i^T v}{\sqrt n}}]
    \le \prod_{i=1}^n 
    \frac{\E e^{\frac{ X_i^Tv}{2\sqrt n}}+\E e^{\frac{-X_i^Tv}{2\sqrt n}}}{2}
    \le e^{L^2 \|\Sigma^{1/2} v\|^2 / 8}
    = e^{L^2 \|u\|^2 / 8},
    $$
    where for the last inequality we use that each $X_i$ is $L$-subgaussian.
    Hence $\tilde Z$ is $L/2$-subgaussian for the logistic loss.
\end{proof}

\lemmaPenalizedLassoBelongsToCone*
\begin{proof}[Proof of \Cref{lemma:penalizedLassoBelongstoCone}]
    We will apply the previous lemma, but first let us derive some properties
    of subgaussian vectors.

    Let $U$ be a random vector valued in $\R^p$ such that
    each component $U_j$ is $1$-subgaussian in the sense that $\E[e^{t U_j^2}] \le e^{t^2/2}$,
    for each $j=1,...,p$. Let $\mu>0$ be a deterministic real.

    Since $U_j$ is $1$-subgaussian, $\P(U_j>\sqrt{2x}) \le e^{-x}$
    by a Chernoff bound,
    hence $(U_j)_+^2/2$ is stochastically dominated by an exponential random variable $\tau$
    with parameter 1 and there exists a probability space on which both $U_j$ and
    $\tau$ are defined such that $U_j\le \sqrt{2\tau}$ holds almost surely.
    Hence using $|\sqrt a - \sqrt b|^2 \le a^2 - b^2$ for any $a>b >0$,
    $$\E[(U_j-\mu)_+^2]
    \le \E[(\sqrt{2\tau}-\mu)_+^2]
    \le \int_{\mu^2/2}^\infty (\sqrt{2t}-\mu)_+^2 e^{-t}dt
    \le 2 \int_{\mu^2/2}^\infty (t - \mu^2/2)_+ e^{-t}dt
    = 2 e^{-\mu^2/2}.
    $$ 
    The same holds with $U_j$ replaced by $-U_j$.
    For $\mu=(1+\xi)\sqrt{2\log(p/s)}$ for $\xi\ge 0$, this shows that
    $$\E\sum_{j=1}^p(|U_j|-\mu)_+^2
    \le \E\sum_{j=1}^p(U_j-\mu)_+^2+(-U_j-\mu)_+^2
    \le 4p e^{-\mu^2/2} = 4s/(p/s)^\xi.
    $$
    By Markov's inequality, for this value of $\mu$ and $\xi>0$,
    \begin{equation}
    \P\left(\frac 1 s \sum_{j=1}^p(|U_j|-\mu)_+^2 \le 2 \xi^2 \log(p/s)\right)
    \ge
    1 - \frac{4}{2\xi^2\log(p/s)(p/s)^\xi}.
        \label{event-h-subgaussian-probability-bound}
    \end{equation}
    By the triangle inequality, on this event, we also have 
    $$\max_{A\subset\{1,...,p\}:|A|=s}
    \left(\frac 1 s \sum_{j\in A}U_j^2\right)^{1/2} \le \mu+\xi\sqrt{2\log(p/s)}=(1+2\xi)\sqrt{2\log(p/s)}.
    $$
    Furthermore, if $\hat A$ denotes the subset achieving the maximum in the left
    hand side above, any value $U_j^2$ with $j\notin \hat A$ is smaller than
    the average of the values in $\hat A$ and
    \begin{equation}
        \max_{j\notin \hat A}|U_j|
        \le
        \max_{A\subset\{1,...,p\}:|A|=s}
        \left(\frac 1 s \sum_{j\in A}U_j^2\right)^{1/2} \le(1+2\xi)\sqrt{2\log(p/s)}.
        \label{event-h-subgaussian}
    \end{equation}

    In linear regression with the squared loss,
    the vector $Z$ is $L\sigma^*$-subgaussian in the sense
    that $\E\exp(u^TZ) \le \exp(L^2(\sigma^*)^2 \|u\|^2/2)$ holds.
    Hence, since $\Sigma$ satisfies the normalization \eqref{eq:Normalized},
    the random vector $U=\Sigma^{1/2} Z / (L\sigma^*)$ satisfies
    for all $j=1,...,p$ that
    $\E[\exp(t U_j)] \le e^{-t^2/2}$  and the bound \eqref{event-h-subgaussian}
    holds on an event of probability at least equal to the right hand side
    of \eqref{event-h-subgaussian-probability-bound}.

    For any $\xi>0$, if $\lambda = L \sigma^*(1+3\xi)\sqrt{2\log(p/s)/n}$
    then any $b\in \hat T$ where $\hat T$ is defined in
    \eqref{set-hat-T-linear-model-squared-loss} satisfies
    $$0\le Z^T\Sigma^{1/2} (b-\beta^*)
    - \lambda\sqrt n(\|b\|_1-\|\beta^*\|_1) 
    = L\sigma^*\left(U^T(b-\beta^*) - \lambda\sqrt n (L\sigma^*)^{-1} (\|b\|_1 - \|\beta^*\|_1) \right).
    $$
    This implies, by replacing $\lambda$ by its value and using the Cauchy-Schwarz inequality on the support of $\beta^*$, that
    $$0 \le U^T\delta + (1+3\xi)\sqrt{2\log(p/s)}(\sqrt s \|\delta\|_2 - \|b_{S^c}\|_1)
    $$
    where $\delta=b-\beta^*$ and $S=\supp(\beta^*)$.
    Then each component of $U$ satisfies $\E[\exp(t U_j)] \le e^{-t^2/2}$ as
    explained above, the bound \eqref{event-h-subgaussian-probability-bound} applies.
    Hereafter, assume that event \eqref{event-h-subgaussian} holds.
    On event \eqref{event-h-subgaussian},
    $\sum_{j\in S} U_j \delta_j \le (1+2\xi)\sqrt{2\log(p/s)} \|\delta\|_2$
    by the Cauchy-Schwarz inequality.
    If $\hat A\subset \{1,...,p\}$ contains the indices of the $s$
    largest coefficients of $U$ in absolute value,
    then
    $\sum_{j\in \hat A} U_j \delta_j \le (1+2\xi)\sqrt{2\log(p/s)} \|\delta\|_2$
    again by the Cauchy-Schwarz inequality.
    Finally, 
    $\sum_{j\notin S\cup\hat A} U_j \delta_j \le (1+2\xi)\sqrt{2\log(p/s)} \|\delta_{S^c}\|_1$
    because  $\max_{j\notin S\cup\hat A} |U_j|$ is bounded from above
    as in \eqref{event-h-subgaussian}.
    Combining the above inequalities, on the event \eqref{event-h-subgaussian}
    we have
    $$0 \le \sqrt{s} \|\delta\|_2 (2 + 5\xi)\sqrt{2\log(p/s)}
    - \xi\sqrt{2\log(p/s)} \|\delta_{S^c}\|_1.
    $$
    This implies that $\|\delta_{S^c}\|_1 \le \sqrt s \|\delta\|_2 (2\xi^{-1}+5)$
    and $\|\delta\|_1 \le \sqrt s \|\delta\|_2(6+2\xi^{-1})$.

    The proof in logistic regression is the same up to a different scaling
    due to $\tilde Z$ from~\Cref{lem:LemmaA1}(ii) being $L/2$-subgaussian,
    while in linear regression with the squared loss we had $Z$ being
    $L\sigma^*$-subgaussian.
\end{proof}

\section{Group-Lasso}

\lemmaGroupLassoBelongsToCone*

\begin{proof}[Proof of \Cref{lemma:GroupLassoBelongstoCone}]
    Define $U=\Sigma^{1/2} Z/(L\sigma^*)$ where $Z$ is as in~\Cref{lem:LemmaA1}(i).
    Then $\E[\exp(v^T Z/(L\sigma^*))] \le \exp(\|v\|^2/2)$
    by the properties of $Z$ stated in~\Cref{lem:LemmaA1}(i).
    We wish to study the restriction $U_{G_k}$ of $U$ to group $G_k$.
    Let $M_k$ be the matrix with $|G_k|$ rows and $p$ columns
    made of the rows of $\Sigma^{1/2}$ indexed in $G_k$.
    Then $U_{G_k}= M_k Z/(L\sigma^*)$ and
by applying the concentration inequality in \cite{hsu2012tail} to the subgaussian
vector $Z/(L\sigma^*)$ and the matrix $M_k^T M_k\in\R^{p\times p}$,
$$\|U_{G_k}\|_2^2 \le \Tr(M_k^T M_k)+2\sqrt{x}\Tr (M_k^T M_k M_k^T M_k)^{1/2} + 2x
\| M_k^T M_k \|_{op}.
$$
with probability at least $1-e^{-x}$.
By properties of the trace,
$\Tr(M_k^T M_k) = \Tr(M_k M_k^T) = \Tr(\Sigma_{G_k,G_k})$.
Similarly for the second term, 
$\Tr (M_k^T M_k M_k^T M_k)^{1/2} = \| \Sigma_{G_k,G_k}\|_F$.
Finally, $\| M_k^T M_k \|_{op} = \|M_k M_k^T\|_{op} = \|\Sigma_{G_k,G_k}\|_{op}$
so that the previous display reads
\begin{align}
\|U_{G_k}\|_2^2 
&\le
\Tr(\Sigma_{G_k,G_k})+2\sqrt{x}\|\Sigma_{G_k,G_k}\|_F + 2x
\| \Sigma_{G_k,G_k} \|_{op}\\
& \le d + 2\sqrt{x d} + 2x \le (\sqrt d + \sqrt{2x})^2
\end{align}
where for the second inequality we used that
$\Tr(\Sigma_{G_k,G_k}) \le |G_k|=d$ and
$\|\Sigma_{G_k,G_k}\|_F^2 \le \sqrt{|G_k|}= \sqrt d$ using the assumption
$\|\Sigma_{G_k,G_j}\|_{op} \le 1$.
Hence $W_k=(\|U_{G_k}\|_2 - \sqrt{d})_+$ is 1-subgaussian for every group $k=1,...,M$,
in the sense that $\P(W_k>\sqrt{2x})\le e^{-x}$.
As previously for the Lasso, $W_k$ is thus stochastically dominated by $\sqrt{2\tau}$ where $\tau$
is an exponential random variable
with parameter 1, and 
\begin{align}
    \label{eq:expectation-group-lasso}
    \E[(\|U_{G_k}\|_2 - \sqrt{d} - \mu)_+^2 ] \le
\E[(\sqrt{2\tau}-\mu)_+^2]
&\le
2 \int_0^\infty (t-\mu^2/2)_+^2 e^{-t}dt = 2e^{-\mu^2/2}.
\end{align}
Define $W\ge 0$ by
$
W^2=\sum_{k=1}^M (\|U_{G_k}\|_2 - \sqrt{d} - \mu)_+^2
$.
We have thus proved that
$\E[W^2] \le 2Me^{-\mu^2/2}$ and for $\mu=(1+\xi)\sqrt{2\log(M/s)}$ we obtain
$\E[W^2] \le 2s/(M/s)^{\xi}$, and by Markov's inequality
\begin{equation}
    \label{eq:above-event-group-lasso}
    \P\left(\tfrac 1 s W^2 \le \xi^2 2\log(M/s)\right) \ge 1 - 2/(\xi^2\log(M/s)(M/s)^\xi).
\end{equation}
Furthermore, on the above event \eqref{eq:above-event-group-lasso}, by the triangle inequality we have
$$\max_{A\subset\{1,...,M\}:|A|= s}\left(\frac 1 s \sum_{k\in A}\|U_{G_k}\|_2^2\right)^{1/2}
\le \sqrt d + \mu+ \xi \sqrt{2\log(M/s)} =\sqrt d + (1+2\xi)\sqrt{2\log(M/s)}
\triangleq \lambda_0.$$
Let $\lambda_0$ be defined as the right hand side of the previous display
and notice that $\sqrt n\lambda/(L\sigma^*)=(1+\xi)\lambda_0$,
so that if $\hat A$ is the subset of $[M]$ with the indices $k$
with largest $\|U_{G_k}\|$ (i.e., a subset attaining the maximum
in the previous display), we have proved that
\begin{equation}
    \label{inequality-lambda-hat-GL}
    \max_{k\in\hat A^c}
    \| (\Sigma^{1/2}Z)_{G_k}\|
    \le
    \Big(\frac 1 s \sum_{k\in \hat A}\|(\Sigma^{1/2}Z)_{G_k}\|_2^2\Big)^{1/2} 
    \le (1+\xi)^{-1}\sqrt n\lambda.
\end{equation}
By~\Cref{lem:LemmaA1}(i),
and the Cauchy-Schwarz inequality on the groups in $S$,
\begin{align*}
    0
    \le
    Z^T\Sigma^{1/2}\delta + \sqrt n \lambda\sum_{k=1}^M (\|\beta^*_{G_k}\| - \|b_{G_k}\|)
    &\le 
    Z^T\Sigma^{1/2}\delta + \sqrt n \lambda\left(\sqrt{s}\|\delta\| - \sum_{k\notin S} \|b_{G_k}\|\right)
    \\&= L\sigma^*
    \left[
    \delta^T U
    + (1+\xi)\lambda_0
    \left(
    \sqrt s \|\delta\|
    -
    \sum_{k\notin S} \|b_{G_k}\|\right)
    \right],
\end{align*}
where $\delta=(b-\beta^*)$
and $U=\Sigma^{1/2}Z/(L\sigma^*)$.
We now bound $U^T\delta$ which appears on the previous display.
On the above event, \eqref{eq:above-event-group-lasso} we have
$\sum_{k\in S} \delta_{G_k}{}^T U_{G_k}
\le(\sum_{k\in S}\|U_{G_k}\|^2)^{1/2} \|\delta\|
\le\sqrt s(\sqrt d  + (1+2\xi)\sqrt{2\log(M/s)})\|\delta\| = \sqrt s \lambda_0\|\delta\|$.
Similarly, if $\hat A$
contains the indices of the $s$ groups with the largest
$\|U_{G_k}\|$ then on the above event \eqref{eq:above-event-group-lasso},
$\sum_{k\in \hat A} \delta_{G_k}{}^T U_{G_k}
\le(\sum_{k\in \hat A}\|U_{G_k}\|^2)^{1/2} \|\delta\|
\le\sqrt s(\sqrt d  + (1+2\xi)\sqrt{2\log(M/s)})\|\delta\|
=\sqrt s \lambda_0\|\delta\|
$.
For any group $G_k$ with $k\notin S\cup \hat A$,
we have $\|U_{G_k}\| \le \lambda_0$.
Combining these bounds with the fact that $\lambda=L\sigma^*(1+\xi)\lambda_0$,
we have established that on event \eqref{eq:above-event-group-lasso},
\begin{equation}
    \label{final-conclusion-lemma-group-lasso-cone}
0\le
(3+\xi)\sqrt s \lambda_0 \|\delta\|
-
\xi \lambda_0 \sum_{k\notin S}\|\delta_{G_k}\|,
\qquad
\sum_{k\notin S}\|\delta_{G_k}\|
\le\sqrt s (1+3/\xi)\|\delta\|
.
\end{equation}
On the groups indexed in $S$, by the Cauchy-Schwarz inequality we have 
$\sum_{k\in S}\|\delta_{G_k}\|
\le\sqrt s \|\delta\|$.
Hence $\delta=b-\beta^*$ belongs to the cone defined in the statement of the
Lemma.
\end{proof}

\lemmaGroupLassoBoundGamma*

\begin{proof}[Proof of \Cref{lemma:GroupLassoBoundGamma}]
    By definition of the restricted eigenvalue $\phi(T)$,
    for any $u\in T$ with $\|\Sigma^{1/2}u\|=1$ we have $\|u\|\le\phi(T)^{-1}$.
    Let $g\sim N(0, I_p)$; we wish to bound the expectation
    $\E\sup_{u\in T: \|\Sigma^{1/2}u\|=1} |g^T\Sigma^{1/2}u|$.
    Let $U=\Sigma^{1/2}g$. 
    Let $\mu=\sqrt{2\log(M/s)}$. We have for any $u\in T$ with $\|\Sigma^{1/2}u\|=1$,
    $$u^TU 
    \le \sum_{k=1}^M \|u_{G_k}\| \|U_{G_k}\|
    = \sum_{k=1}^M \|u_{G_k}\|(\|U_{G_k}\| - \mu - \sqrt d)
    + (\mu+\sqrt d)\sum_{k=1}^M \|u_{G_k}\|.
    $$
    For the second term, since $u\in T$, inequality $\sum_{k=1}^M \|u_{G_k}\|\le \sqrt s (2+3/\xi)\|u\|$ hence
    the second term is bounded from above as follows,
    $(\mu+\sqrt d)\sum_{k=1}^M \|u_{G_k}\| \le (\mu+\sqrt d)\sqrt s(2+3/\xi) \phi(T)^{-1}$.

    It remains to bound the first term. By the Cauchy-Schwarz inequality,
    $$\sum_{k=1}^M \|u_{G_k}\|(\|U_{G_k}\| - \mu - \sqrt d)
    \le \|u\|\left(\sum_{k=1}^M (\|U_{G_k}\|-\mu-\sqrt d)_+^2 \right)^{1/2}.
    $$
    Finally, $\|u\|\le\phi(T)^{-1}$ and the expectation bound \eqref{eq:expectation-group-lasso} show that the previous display is bounded from above
    by $\phi(T)^{-1} (M 2e^{-\mu^2/2})^{1/2} = \phi(T)^{-1} \sqrt{2 s}$.

\end{proof}

\section{Proofs of \Cref{sec:ExactRiskBound}}

\theoremExactRiskIdentity*

\begin{proof}[Proof of Theorem~\ref{thm:ExactRiskBound}]
Since $(X_i)_{i=1}^n$ and $(\varepsilon_i)_{i=1}^n$ are independent, we will condition throughout on $\varepsilon_1,\ldots,\varepsilon_n$. The proximal operator is Lipschitz
for any convex function $h$, in the sense that
\[
\|\prox_h(\beta^* + z) - \prox_h(\beta^* + z')\| \le \|z - z'\|\quad\mbox{for all}\quad z,z'\in\mathbb{R}^p.
\]
see Definition 2.3 and the following discussion in \cite{lee2014proximal}.
If $\Sigma=I_p$, the first order expansion $\eta$ in \eqref{eq:FirstOrderApproxRegularized} is given by
$\eta = \prox_h\left(\beta^* + n^{-1}\sum_{i=1}^n X_i\varepsilon_i\right)$.
This means $\eta = \prox_h(\beta^* + n^{-1/2}\sigma^*Z)$ for $Z=(\sum_{i=1}^n\eps_i^2)^{-1/2} \sum_{i=1}^n \eps_iX_i$ and
$Z\sim N(0, I_p)$ if $X_1,...,X_n$ are iid $N(0,I_p)$ independent of $\eps_1,...,\eps_n$.
Note that in this case, $Z$ is independent of $\sigma^*$.
Thus $\eta$ is a $n^{-1/2}\sigma^*$-Lipschitz function of $Z$,
and by the triangle inequality $\|\eta-\beta^*\|$ is also a $n^{-1/2}\sigma^*$-Lipschitz
function of $Z$.
By the Gaussian concentration inequality \cite[Theorem 5.6]{boucheron2013concentration},
with probability $1 - 2\exp(-t^2/2)$ we have
\[
|\|\eta - \beta^*\| - \mathbb{E}_Z\left[\|\eta - \beta^*\|\right]| \le  n^{-1/2}\sigma^*t.
\]
The Gaussian Poincar\'e inequality \cite[Theorem 3.20]{boucheron2013concentration}
implies that $(\mbox{Var}(\|\eta - \beta^*\|))^{1/2}$ is bounded by the Lipschitz constant and hence
$|\mathbb{E}_Z\left[\|\eta - \beta^*\|\right] - (\mathbb{E}_Z[\|\eta - \beta^*\|^2])^{1/2}| \le n^{-1/2}\sigma^*.$
Combining these inequalities above, we get with probability $1 - 2\exp(-t^2/2)$
\[
\left|\|\eta - \beta^*\| - (\mathbb{E}[\|\beta^* - \prox_h(\beta^* + n^{-1/2}\sigma^*Z)\|^2])^{1/2}\right| \le n^{-1/2}\sigma^*(t + 1).
\]
Therefore, by triangle inequality, on the same event we have
\[
\left|\|\hbeta - \beta^*\| - (\mathbb{E}[\|\beta^* - \prox_h(\beta^* + n^{-1/2}\sigma^*Z)\|^2])^{1/2}\right| \le n^{-1/2}\sigma^*(t + 1) + \|\eta - \hbeta\|
\]
which completes the proof.
\end{proof}

\section{Proofs of \Cref{sec:Inference}}
\propositionDebiasing*

\begin{proof}[Proof of \Cref{prop:debiasing}]
    Some of the argument below is borrowed from \cite[Section
    6]{bellec_zhang2019debiasing_adjust}.
    Let $T_n = \sqrt n \|z_a\|^{-2} z_a^T\eps$,
    and let $Q_a = I_p - \Sigma^{-1} a a^T$; notice that $T_n$
    has the $t$-distribution with $n$ degrees-of-freedom.
    Also note that $\Sigma^{1/2}Q_a\Sigma^{-1/2} = I - (\Sigma^{-1/2}a)(\Sigma^{-1/2}a)^{\top}$ and $z_a = X\Sigma^{-1}a = X\Sigma^{-1/2}\Sigma^{-1/2}a$; this implies that
    $\X Q_a$ is independent of $z_a$ because $(z_a, \X Q_a)$
    are jointly normal and uncorrelated.
    By definition of $\hat\theta$, simple algebra yields
    that 
    \begin{align}
    \sqrt n (\hat \theta - a^T\beta^*) - T_n
    &= - \sqrt n \|z_a\|^{-2} z_a^T \X Q_a(\hbeta-\beta^*) 
    \\
    &= - \sqrt n \|z_a\|^{-2} z_a^T \X Q_a(\eta -\beta^*) 
    + \sqrt n \|z_a\|^{-2} z_a^T \X Q_a(\eta - \hbeta).
    \end{align}
    Note that $\eta$ only depends on $\X$ through $\eps^T\X$
    hence $\eta$ is independent of $(P_\eps^\perp \X,P_\eps^\perp z_a)$
    where $P_\eps^\perp = I_n - \|\eps\|^{-2}\eps\eps^T$ is an orthogonal projection;
    set also $P_\eps = \|\eps\|^{-2}\eps\eps^T$ for the complementary projection.
    We further split the first term above, so that $\sqrt n(\hat \theta - a^T\beta^*)-T_n$
    is equal to
    $$
\sqrt n \|z_a\|^{-2}
\left(
    -
    \left[z_a^T P_\eps^\perp\X Q_a(\eta -\beta^*)\right]
    -\left[z_a^T P_\eps \X Q_a(\eta -\beta^*) \right]
    + \left[ z_a^T \X Q_a(\eta - \hbeta)\right] 
    \right)
    $$
    Hereafter, assume that the event $\{\eta-\beta^*,\eta-\hbeta\}\subset T$ holds
    (this holds with probability at least $1-\alpha$ by assumption).
    For the first bracket inside the large parenthesis,
    $P_\eps^\perp z_a$ is independent of $(\eta,\X Q_a)$
    conditionally on $\eps$, so that
    $z_a^T P_\eps^\perp\X Q_a(\eta -\beta^*)= O_p(1) \|\X Q_a(\eta-\beta^*)\|$.
    For the second bracket, since $P_\eps$ is rank 1, $\|P_\eps z_a\| = O_p(1)$
    and the second bracket is also $O_p(1) \|\X Q_a(\eta-\beta^*)\|$
    by the Cauchy-Schwarz inequality.
    Since $\eta-\beta^*\in T$ and $r_n=\gamma(T,\Sigma)/\sqrt n$ we have
    $\X = \X Q_a + z_a a^T$ so that by the triangle inequality,
    $$\|\X Q_a(\eta-\beta^*)\|
    \le \|\X(\eta-\beta^*)\| + \|z_a\| |a^T(\eta-\beta^*) |
    \le \|\X(\eta-\beta^*)\| + \|z_a\| \|\Sigma^{1/2}(\eta-\beta^*)\|.
    $$
    By an application of \cite{plan_vershynin_liaw2017simple},
    $\sup_{u\in T: \|\Sigma^{1/2}u\|=1}|\|\X u\| - \sqrt n \|\Sigma^{1/2}u\||
    \le \C(\gamma(T,\Sigma) + t)$
    with probability at least $1-2e^{-t^2/2}$,
    since $\X \Sigma^{-1/2}$ has iid $N(0,1)$ entries.
    We take $t = \gamma(T,\Sigma)=r_n\sqrt n$. Since $\|z_a\| = O_p(\sqrt n)$,
    the first two brackets above are 
    $O_p(\sqrt n (1+r_n)) \|\Sigma^{1/2}(\eta-\beta^*)\|$.

    We now focus on the third bracket. Since $\eta-\hbeta\in T$,
    \begin{equation}
        \label{eq:proof-rhs}
        z_a^T\X Q_a (\eta-\hbeta) \le 
        \|\Sigma^{1/2}(\eta-\hbeta)\| \|z_a\|
        \sup_{u \in T: \|\Sigma^{1/2}u\|=1}\|z_a\|^{-1} z_a^T \X Q_a \Sigma^{-1/2} \Sigma^{1/2}u.
    \end{equation}
    Since $z_a$ and $\X Q_a$ are independent,
    conditionally on $z_a$, the random vector
    $\|z_a\|^{-1}z_a^T \X Q_a \Sigma^{-1/2}$
    is normal
    with covariance matrix $ \Sigma^{1/2} Q_a \Sigma^{-1/2}$.
    Since $\Sigma^{1/2} Q_a \Sigma^{-1/2}$ is a projection matrix in $\R^p$,
    $\Sigma^{1/2} Q_a \Sigma^{-1/2} \preceq I_p$ in the sense of positive semi-definite
    matrices. By 
    the Sudakov-Fernique’s inequality \cite[Theorem 7.2.11]{vershynin2018high},
    conditionally on $z_a$ we have
    $$\E\left[\sup_{u \in T: \|\Sigma^{1/2}u\|=1} \|z_a\|^{-1} z_a^T \X Q_a \Sigma^{-1/2} \Sigma^{1/2} u  \;\Big|\; z_a \right]
    \le
    \E\left[\sup_{u \in T: \|\Sigma^{1/2}u\|=1} g^T \Sigma^{1/2} u \right]
    =\gamma(T,\Sigma)
    $$
    for some $g\sim N(0,I_p)$.
    Furthermore, by Gaussian concentration \cite[Theorem 5.8]{boucheron2013concentration}
    again conditionally on $z_a$ we have
    $\sup_{u \in T: \|\Sigma^{1/2}u\|=1} \|z_a\|^{-1} z_a^T \X Q_a \Sigma^{-1/2} \Sigma^{1/2} u \le \gamma(T,\Sigma) + t$
    with probability at least $1-e^{-t^2/2}$.
    Taking $t= \gamma(T,\Sigma) = r_n \sqrt n$, we obtain that the right hand
    side of \eqref{eq:proof-rhs} is bounded from above
    on an event of probability at least $1-e^{-t^2/2}$ by
    $\|\Sigma^{1/2}(\eta-\hbeta)\| \|z_a\| \sqrt n  r_n$.
    Given that $\|z_a\|/\sqrt n = O_p(1)$ and $\sqrt n /\|z_a\| = O_P(1)$,
    the proof is complete.
\end{proof}

\section{Proofs of Results in~\ref{sec:proof-sketch}}
\label{appendix:proof-sketch-details}

\theoremSketchDeterministic*

\begin{proof}[Proof of Theorem~\ref{thm:Deterministic}]
    \label{proof:thm:GeneralLossSlow}
By strong convexity of the objective function of $\eta$ with respect to
the norm $\|\cdot\|_K$, or alternatively by application
of for instance \cite[Lemma 1]{bellec2018prediction} or
or \cite[Lemma A.2]{bellec2016slope},
\begin{align}
\frac{1}{2}\|K^{1/2}(\hat{\beta} - \eta)\|^2 &\le \frac{1}{n}\sum_{i=1}^n \ell'(Y_i, X_i^{\top}\beta^*)X_i^{\top}(\hat{\beta} - \beta^*) + \frac{1}{2}\|K^{1/2}(\hat{\beta} - \beta^*)\|^2 + h(\hat{\beta})\label{eq:OptimalityEta}\\
&\qquad\quad - \left[\frac{1}{n}\sum_{i=1}^n \ell'(Y_i, X_i^{\top}\beta^*)X_i^{\top}(\eta - \beta^*) + \frac{1}{2}\|K^{1/2}(\eta - \beta^*)\|^2 + h(\eta)\right]. \nonumber
\end{align}
From the definition~\eqref{hbeta-intro} of $\hat{\beta}$, we get
\begin{equation}\label{eq:OptimalityHatBeta}
0 \le \frac{1}{n}\sum_{i=1}^n \ell(Y_i, X_i^{\top}\eta) + h(\eta) - \frac{1}{n}\sum_{i=1}^n \ell(Y_i, X_i^{\top}\hat{\beta}) - h(\hat{\beta}).
\end{equation}
Adding inequalities~\eqref{eq:OptimalityEta} and~\eqref{eq:OptimalityHatBeta}, 
the terms $h(\hbeta)$ and $h(\eta)$ cancel out and we obtain
\begin{equation}\label{eq:CombinedOptimalityIneq}
\begin{split}
&\frac{1}{2}\|K^{1/2}(\hat{\beta} - \eta)\|^2\\ &\qquad\le \frac{1}{n}\sum_{i=1}^n \ell(Y_i, X_i^{\top}\eta) - \frac{1}{n}\sum_{i=1}^n \ell'(Y_i, X_i^{\top}\beta^*)X_i^{\top}(\eta - \beta^*) - \frac{1}{2}\|K^{1/2}(\eta - \beta^*)\|^2\\
&\qquad\quad - \left[\frac{1}{n}\sum_{i=1}^n \ell(Y_i, X_i^{\top}\hat{\beta}) - \frac{1}{n}\sum_{i=1}^n \ell'(Y_i, X_i^{\top}\beta^*)X_i^{\top}(\hat{\beta} - \beta^*) - \frac{1}{2}\|K^{1/2}(\hat{\beta} - \beta^*)\|^2\right].
\end{split}
\end{equation}
The terms on the right hand side resemble the remainder terms from a second order Taylor series expansion except for $K$ instead of $\hat{K}$. Using the Taylor expansion above, we get for any $\beta\in\mathbb{R}^p$
\begin{align}
&\frac{1}{n}\sum_{i=1}^n \ell(Y_i, X_i^{\top}\beta) - \frac{1}{n}\sum_{i=1}^n \ell'(Y_i, X_i^{\top}\beta^*)X_i^{\top}(\beta - \beta^*) - \frac{1}{2}\|K^{1/2}(\beta - \beta^*)\|^2\\
&\qquad= \frac{1}{2n}\sum_{i=1}^n |X_i^{\top}(\beta - \beta^*)|^2a_i(\beta) + \frac{1}{2}(\beta - \beta^*)^{\top}(\hat{K} - K)(\beta - \beta^*),  
\end{align}
where
\begin{equation}\label{eq:DefAi}
a_i(\beta) := \int_0^1 \{\ell''(Y_i, X_i^{\top}\beta^* + tX_i^{\top}(\beta - \beta^*)) - \ell''(Y_i, X_i^{\top}\beta^*)\}dt.
\end{equation}
(Note that for the squared loss, $\ell''=1$ is constant and $a_i(\cdot)=0$
which leads to a mucher simple analysis;
the reader only interested in squared loss may skip the analysis of $a_i(\cdot)$).
Substituting this in~\eqref{eq:CombinedOptimalityIneq}, we get
\begin{align}
\|K^{1/2}(\hat{\beta} - \eta)\|^2 &\le (\eta - \beta^*)^{\top}(\hat{K} - K)(\eta - \beta^*)
- (\hat{\beta} - \beta^*)^{\top}(\hat{K} - K)(\hat{\beta} - \beta^*)
\label{eq:LeadFastRate-1}
\\
&\quad+ \frac{1}{n}\sum_{i=1}^n |X_i^{\top}(\eta - \beta^*)|^2a_i(\eta) - \frac{1}{n}\sum_{i=1}^n |X_i^{\top}(\hat{\beta} - \beta^*)|^2a_i(\hat{\beta}).
\label{eq:LeadFastRate-2}
\end{align}
From the Lipschitz condition on $\ell''(\cdot, \cdot)$, we get $|a_i(\beta)| \le B|X_i^{\top}(\beta - \beta^*)|,$ and hence part 1 of the result follows.

In part 1, we did not use any information about $\hat{\beta}-\eta$. For part 2, we will control the right hand side ``quadratic forms'' in~\eqref{eq:LeadFastRate-1} in a more refined way. By simple algebra and the definition of $Q_{n,2}(\cdot)$,
\begin{equation}\label{eq:QuadraticControlDiff}
\begin{split}
&(\eta - \beta^*)^{\top}(\hat{K} - K)(\eta - \beta^*) - (\hat{\beta} - \beta^*)^{\top}(\hat{K} - K)(\hat{\beta} - \beta^*)\\ &\quad= ((\eta - \beta^*) - (\hat{\beta} - \beta^*))^{\top}(\hat{K} - K)((\eta - \beta^*) + (\hat{\beta} - \beta^*))\\
&\quad= (\eta - \hat{\beta})^{\top}(\hat{K} - K)(\eta - \beta^*)\\
&\quad\qquad+ (\eta - \hat{\beta})^{\top}(\hat{K} - K)(\hat{\beta} - \beta^*)\\
&\quad\le Q_{n,2}(T)\|K^{1/2}(\eta - \hat{\beta})\|\left[\|K^{1/2}(\hat{\beta} - \beta^*)\| + \|K^{1/2}(\eta - \beta^*)\|\right].
\end{split}
\end{equation}
This completes the control of first difference in~\eqref{eq:LeadFastRate-1}. For the second difference in~\eqref{eq:LeadFastRate-2}, observe that 
\[
|a_i(\beta)| \le B|X_i^{\top}(\beta - \beta^*)|\quad\mbox{and}\quad |a_i(\beta) - a_i(\alpha)| \le B|X_i^{\top}(\beta - \alpha)|,
\]
and hence
\begin{align*}
&\frac{1}{n}\sum_{i=1}^n |X_i^{\top}(\eta - \beta^*)|^2a_i(\eta) - \frac{1}{n}\sum_{i=1}^n |X_i^{\top}(\hat{\beta} - \beta^*)|^2a_i(\hat{\beta})\\
&\quad= \frac{1}{n}\sum_{i=1}^n a_i(\eta)\left[|X_i^{\top}(\eta - \beta^*)|^2 - |X_i^{\top}(\hat{\beta} - \beta^*)|^2\right]+ \frac{1}{n}\sum_{i=1}^n |X_i^{\top}(\hat{\beta} - \beta^*)|^2\left[a_i(\eta) - a_i(\hat{\beta})\right]\\
&\quad{\le} \frac{B}{n}\sum_{i=1}^n |X_i^{\top}(\eta - \beta^*)|\times|X_i^{\top}((\eta - \beta^*) - (\hat{\beta} - \beta^*))|\times|X_i^{\top}((\eta - \beta^*) + (\hat{\beta} - \beta^*))|\\
&\quad\quad+ \frac{B}{n}\sum_{i=1}^n |X_i^{\top}(\hat{\beta} - \beta^*)|^2\times|X_i^{\top}(\eta - \hat{\beta})|\\
&\quad\le \frac{B}{n}\sum_{i=1}^n |X_i^{\top}(\eta - \hat{\beta})|\times\left[|X_i^{\top}(\eta - \beta^*)|^2 + |X_i^{\top}(\hat{\beta} - \beta^*)|^2\right]\\
&\quad\quad+ \frac{B}{n}\sum_{i=1}^n |X_i^{\top}(\eta - \hat{\beta})|\times|X_i^{\top}(\eta - \beta^*)|\times|X_i^{\top}(\hat{\beta} - \beta^*)|.
\end{align*}
The right hand side is trivially bounded by
\[
3B\|\eta - \hat{\beta}\|_{K}\left[\|(\eta - \beta^*)\|_{K}^2 + \|(\hat{\beta} - \beta^*)\|_{K}^2\right]\sup_{u,v,w\in \tilde T}\frac{1}{n}\sum_{i=1}^n\frac{|X_i^{\top}u|}{\|K^{1/2}u\|}\times\frac{|X_i^{\top}v|}{\|K^{1/2}v\|}\times\frac{|X_i^{\top}w|}{\|K^{1/2}w\|}.
\]
Using $3 abc \le a^3+b^3+c^3$ for any positive $\{a,b,c\}$,
the previous display is bounded from above by
\[
B\|K^{1/2}(\eta - \hat{\beta})\|\left[\|K^{1/2}(\eta - \beta^*)\|^2 + \|K^{1/2}(\hat{\beta} - \beta^*)\|^2\right]Z_n(T).
\]
Substituting these bounds in~\eqref{eq:LeadFastRate-2}, we get the result.
\end{proof}



\propositionControlProc*

\begin{proof}[Proof of \Cref{prop:ControlProc}]
    \label{proof:prop:ControlQn}
Define the function classes $F$ and $H$ as
\begin{align*}
F &:= \left\{(x, y)\mapsto {x^{\top}u}/{\|u\|_K}:\,u\in T\right\},\\
H &:= \left\{(x, y)\mapsto {\ell''(y, x^{\top}\beta^*)x^{\top}u}/{\|u\|_K}:\,u\in T\right\}.
\end{align*}
It is then clear that
\begin{equation*}
\begin{split}
\max\{Q_{n,1}(T), Q_{n,2}(T)\} &\le \sup_{f\in F, h\in H}\,\left|\frac{1}{n}\sum_{i=1}^n \left\{f(X_i,Y_i)h(X_i, Y_i) - \mathbb{E}[f(X_i)h(X_i)]\right\}\right|,\\
Z_n(T) &= \sup_{f\in\mathcal{F}}\,\frac{1}{n}\sum_{i=1}^n |f(X_i)|^3.
\end{split}
\end{equation*}
We now apply Theorem 1.13 of~\cite{mendelson2016upper} with $2^{s_0/2} = \gamma(T, \Sigma), q = 5$. Hence, we get for any $t \ge 8$ with probability at least $1 - 2\exp(-c_1 t^2\gamma^2(T,\Sigma))$,
\begin{align}
&\sup_{f\in F, h\in H}\,\left|\frac{1}{n}\sum_{i=1}^n \left\{f(X_i,Y_i)h(X_i, Y_i) - \mathbb{E}[f(X_i, Y_i)h(X_i,Y_i)]\right\}\right|\label{eq:ProductProcessBound}\\
&\qquad\le \frac{c_2t^2}{n}(\gamma(F) + 2^{s_0/2}\mbox{diam}(F))(\gamma(H) + 2^{s_0/2}\mbox{diam}(H))\nonumber\\
&\qquad\quad+ \frac{c_2t}{\sqrt{n}}(\mbox{diam}(F)(\gamma(H) + 2^{s_0/2}\mbox{diam}(H)) + \mbox{diam}(H)(\gamma(F) + 2^{s_0/2}\mbox{diam}(F))),\nonumber
\end{align}
where
\begin{align*}
\gamma(F) &:= L\mathbb{E}\left[\sup_{u\in T}\frac{|g^{\top}\Sigma^{1/2}u|}{\|K^{1/2}u\|}\right] \le L\sup_{u\in T}\frac{\|\Sigma^{1/2}u\|}{\|K^{1/2}u\|}\gamma(T, \Sigma) \le B_3^{1/2}L\gamma(T, \Sigma),\\
\gamma(H) &:= B_2L\mathbb{E}\left[\sup_{u\in T}\frac{|g^{\top}\Sigma^{1/2}u|}{\|K^{1/2}u\|}\right] \le B_2B_3^{1/2}L\gamma(T, \Sigma),
\end{align*}
and
\begin{align*}
\mbox{diam}(F) &:= L\sup_{u\in T}\frac{\|\Sigma^{1/2}u\|}{\|K^{1/2}u\|} \le B_3^{1/2}L,\\
\mbox{diam}(H) &:= B_2L\sup_{u\in T}\frac{\|\Sigma^{1/2}u\|}{\|K^{1/2}u\|} \le B_2B_3^{1/2}L.
\end{align*}
Substituting these quantities in~\eqref{eq:ProductProcessBound}, part 1 of the result follows. 
Alternatively, one could apply Theorem 1.12 of~\cite{mendelson2016upper} if $\ell(\cdot, \cdot)$ is assumed to be convex in the second argument (implying $\ell''(y, x^{\top}\beta^*) \ge 0$).

To prove part 2, we apply Equation (3.9) of~\cite{mendelson2010empirical} with $|I| = n$. Hence, we have with probability $1 - 2\exp(-c_1t\log n)$
\begin{equation}\label{eq:ZnBound}
Z_n(T) = \sup_{f\in F}\frac{1}{n}\sum_{i=1}^n |f(X_i)^3| \le \frac{c t^3}{n}\left(\gamma(F) + n^{1/3}\mbox{diam}(F)\right)^3,
\end{equation}
where $c>0$ is an absolute constant and
\[
\gamma(F) := L\mathbb{E}\left[\sup_{u\in T}\frac{|g^{\top}\Sigma^{1/2}u|}{\|K^{1/2}u\|}\right] \le B_3^{1/2}L\gamma(T,\Sigma),
\]
and
\[
\mbox{diam}(F) := L\sup_{u\in T}\frac{\|\Sigma^{1/2}u\|}{\|K^{1/2}u\|} \le B_3^{1/2}L.
\]
Hence part 2 follows.
\end{proof}
\subsection{Verification of Assumption~\ref{eq:LossAssump} for Logistic Loss}

\propositionLogisticSetting*

Lipschitzness and boundedness of $\ell''(y, u)$ for logistic loss is straightforward. These parts do not require $\|\Sigma^{1/2}\beta^*\| \le 1$. In order to prove the third part, we prove the following general result for general loss function with a lower second-order curvature.

Define for any $t > 0$
\[
\alpha(t) := \inf_{y\in\mathbb{R}}\inf_{|u| \le t}\,\ell''(y, u).
\]
Note that $\alpha(\cdot)$ is non-increasing. This is called the lower curvature function and appears in the works~\citet[Section 3.2]{loh2017statistical} and~\citet[Section 3]{koltchinskii2009sparsity}.
\begin{proposition}\label{prop:EigenvalueCondition}
Suppose $X_1, \ldots, X_n$ are iid $L$-subGaussian random vectors with covariance $\Sigma$. Then there exists absolute constants $c_1, c_2 > 0$ such that for any $u\in\mathbb{R}^p$, we have
\begin{align}
\frac{u^{\top}K_nu}{u^{\top}\Sigma u} 
~&\ge~ \sup_{\tau > 0}\,\alpha(\tau)\left\{1 - c_1L(\mathbb{P}(|X_1^{\top}\beta^*| > \tau))^{1/2}\right\}\\ 
~&\ge~ \alpha(2L\sqrt{\log(c_2L)}\|\Sigma^{1/2}\beta^*\|_2) / 2.
\end{align}
\end{proposition}
\begin{proof}
Fix a number $\tau > 0$. It is clear that
\begin{equation}\label{eq:TruncationBoundCovariance}
u^{\top}Ku = \mathbb{E}\left[\ell''(Y_1, X_1^{\top}\beta^*)(X_1^{\top}u)^2\right] \ge \alpha(\tau)\mathbb{E}\left[(X_1^{\top}u)^2\mathbbm{1}\{|X_1^{\top}\beta^*| \le \tau\}\right]. 
\end{equation}
Observe now that for any $u\in\mathbb{R}^p$, we have
\begin{align*}
0 &\le \mathbb{E}[(X_1^{\top}u)^2] - \mathbb{E}[(X_1^{\top}u)^2\mathbbm{1}\{|X_1^{\top}\beta^*| \le \tau\}]\\
&= \mathbb{E}[(X_1^{\top}u)^2\mathbbm{1}\{|X_1^{\top}\beta^*| > \tau\}]\\
&\le (\mathbb{E}[(X_1^{\top}u)^4])^{1/2}(\mathbb{P}(|X_1^{\top}\beta^*| > \tau))^{1/2}\\
&\le cL\|\Sigma^{1/2}u\|^2(\mathbb{P}(|X_1^{\top}\beta^*| > \tau))^{1/2},
\end{align*}
for some absolute constant $c > 0$. This implies that
\[
\mathbb{E}\left[(X_1^{\top}u)^2\mathbbm{1}\{|X_1^{\top}\beta^*| \le \tau\}\right] \ge \|\Sigma^{1/2}u\|_2^2\left\{1 - cL(\mathbb{P}(|X_1^{\top}\beta^*| > \tau))^{1/2}\right\}.
\]
Combining this with~\eqref{eq:TruncationBoundCovariance}, we get
\[
\frac{u^{\top}Ku}{u^{\top}\Sigma u} ~\ge~ \sup_{\tau > 0}\,\alpha(\tau)\left\{1 - cL(\mathbb{P}(|X_1^{\top}\beta^*| > \tau))^{1/2}\right\}.
\]
This proves the first inequality. To prove the second inequality, take $\tau = \rho\|\Sigma^{1/2}\beta^*\|_2$ for some $\rho > 0$ (to be determined later). For this choice, we have from $L$-subGaussianity
\[
\mathbb{P}\left(|X_1^{\top}\beta^*| \ge \tau\right) = \mathbb{P}(|X_1^{\top}\beta^*| \ge \rho\|\Sigma^{1/2}\beta^*\|_2) \le 2\exp\left(-\frac{\rho^2}{2L^2}\right).
\]
Hence, if $\rho = 2L(\log(\sqrt{8}cL))^{1/2}$ then
\begin{align}
1 - cL(\mathbb{P}(|X_1^{\top}\beta^*| > \tau))^{1/2} &\ge 1 - \sqrt{2}cL\exp\left(-\frac{\rho^2}{4L^2}\right)\\ &= 1 - \exp\left(-\frac{\rho^2}{4L^2} + \log(\sqrt{2}cL)\right) \ge \frac{1}{2}.
\end{align}
Therefore, for any $u\in\mathbb{R}^p$,
\[
\frac{u^{\top}Ku}{u^{\top}\Sigma u} ~\ge~ \frac{1}{2}\alpha(2L(\log(\sqrt{8}cL))^{1/2}\|\Sigma^{1/2}\beta^*\|_2).
\]
This completes the proof.
\end{proof}
\section{Verification of Restricted Strong Convexity and Rates for Logistic Lasso}\label{sec:RSCLogistic}
In the main paper, we proved/stated bounds for $\|\hat{\beta} - \beta^*\|_K$ and $\|\eta - \beta^*\|_K$ for squared loss with different penalties. These proofs can be extended to the case of logistic loss once restricted strong convexity (RSC) condition is verified; see~\Cref{prop:RateLogisticLasso-mainpaper} which is restated and proved below. Also, see~\cite{negahban2009unified} where the RSC was introduced. We present the following result that proves RSC for general cones $T$.
\begin{proposition}\label{prop:RSCVerifiedC}
Fix any cone $T\subset\mathbb{R}^p$. Assume~\ref{eq:LossAssump} and~\ref{eq:SubGaussianAssump} holds. If the loss function satisfies
\begin{equation}\label{eq:Stability2Deriv}
\sup_{|s - t| \le u}\,\frac{\ell''(y, s)}{\ell''(y, t)} \le \exp(3u),
\end{equation}
then for any $u\in T$ satisfying $\|u\|_{K} \le 1$, we have
\begin{equation}
\begin{split}
\frac{1}{n}\sum_{i=1}^n \ell(Y_i, X_i^{\top}\beta^* + X_i^{\top}u) &- \frac{1}{n}\sum_{i=1}^n \ell(Y_i, X_i^{\top}\beta^*) -  \frac{1}{n}\sum_{i=1}^n u^{\top}X_i\ell'(Y_i, X_i^{\top}\beta^*)\\ 
&\ge \|u\|_{K}^2\frac{B^{-1}_2[0.5B_3^{-1/2} - \tilde{Q}_n(T)]}{\exp(6\sqrt{2}L(\log(4B_3^{1/2}B_2L^2))^{1/2})},
\end{split}
\end{equation}
for a random quantity $\tilde{Q}_n(T)$ which satisfies the following: there exists a universal constant $C> 0$ such that for any $t\ge 0$, we have with probability $1 - \exp(-t)$,
\[
Q_n(T) \le C\left(\frac{B_2L\gamma(T,\Sigma)}{\sqrt{n}} + \frac{B_2\gamma^2(T,\Sigma)}{n} + \frac{t^{1/2}B_2L^2}{\sqrt{n}} + \frac{tB_2L^2}{n}\right).
\]
\end{proposition}
Assumption~\eqref{eq:Stability2Deriv} can be verified for logistic regression easily. Therefore, if $\gamma(T,\Sigma)/\sqrt{n}\to0$ then for all $u\in T, \|u\|_{K} \le 1$, we get
\[
\frac{1}{n}\sum_{i=1}^n \ell(Y_i, X_i^{\top}\beta^* + X_i^{\top}u) - \frac{1}{n}\sum_{i=1}^n \ell(Y_i, X_i^{\top}\beta^*) -  \frac{1}{n}\sum_{i=1}^n u^{\top}X_i\ell'(Y_i, X_i^{\top}\beta^*) \ge \mathfrak{C}\|u\|_{\Sigma}^2,
\]
for a constant $\mathfrak{C}$ depending on $L, B, B_2, B_3$.
\begin{proof}
For notational convenience, let
\[
f_n(\beta) := \frac{1}{n}\sum_{i=1}^n \ell(Y_i, X_i^{\top}\beta).
\]
Define for $\beta^*$ and any $u\in T$,
\[
\Delta_2f_n(\beta^*, u) := f_n(\beta^* + u) - f_n(\beta^*) - u^{\top}\nabla f_n(\beta^*).
\]
Note that
\[
\Delta_2f_n(\beta^*, u) = \frac{1}{n}\sum_{i=1}^n \ell''(Y_i, X_i^{\top}(\beta^* + su))\langle X_i, u\rangle^2,
\]
for some $s\in[0, 1]$. From the stability property~\eqref{eq:Stability2Deriv} of $\ell''(\cdot, \cdot)$, we get that
\[
\ell''(Y_i, X_i^{\top}(\beta^* + su))\langle X_i, u\rangle^2 \ge \exp(-3|\langle u, X_i\rangle|)\ell''(Y_i, X_i^{\top}\beta^*)\langle X_i, u\rangle^2.
\]
This implies that
\begin{equation}
\begin{split}
\Delta_2f_n(\beta^*, u) &= \frac{1}{n}\sum_{i=1}^n \ell''(Y_i, X_i^{\top}\beta^*)\exp(-3|\langle u, X_i\rangle|)\langle X_i, u\rangle^2\\
&= \|u\|_{\Sigma}^2\frac{1}{n}\sum_{i=1}^n \ell''(Y_i, X_i^{\top}\beta^*)\exp(-3|\langle u, X_i \rangle|)\left(\frac{\langle u, X_i\rangle}{\|u\|_{\Sigma}}\right)^2
\end{split}
\end{equation}
Now define the function
\[
\varphi_{\tau}(t) = \begin{cases}|t|,&\mbox{if }|t| \le \tau/2,\\
\tau - |t|,&\mbox{if }\tau/2 \le |t| \le \tau,\\
0, &\mbox{otherwise.}\end{cases}
\]
It is clear that
\[
\frac{\langle X_i, u\rangle^2}{\|u\|_{\Sigma}^2} ~\ge~ \varphi_{\tau}^2\left(\frac{\langle X_i, u\rangle}{\|u\|_{\Sigma}}\right).
\]
Further $\varphi_{\tau}(\cdot)$ is a $1$-Lipschitz function. Using these properties, we get
\begin{equation}
\begin{split}
\Delta_2f_n(\beta^*, u) &\ge \|u\|_{\Sigma}^2\frac{1}{n}\sum_{i=1}^n \ell''(Y_i, X_i^{\top}\beta^*)\exp(-3|\langle u, X_i\rangle|)\varphi_{\tau}^2\left(\frac{\langle X_i, u\rangle}{\|u\|_{\Sigma}}\right)\\
&\ge \|u\|_{\Sigma}^2\exp(-3\tau\|u\|_{\Sigma})\frac{1}{n}\sum_{i=1}^n \ell''(Y_i, X_i^{\top}\beta^*)\varphi_{\tau}^2\left(\frac{\langle X_i, u\rangle}{\|u\|_{\Sigma}}\right).
\end{split}
\end{equation}
To complete the proof note that
\begin{equation}
\begin{split}
&\frac{1}{n}\sum_{i=1}^n \ell''(Y_i, X_i^{\top}\beta^*)\varphi_{\tau}^2\left(\frac{\langle X_i, u\rangle}{\|u\|_{\Sigma}}\right)\\ 
&\qquad= \frac{1}{n}\sum_{i=1}^n \mathbb{E}\left[\ell''(Y_i, X_i^{\top}\beta^*)\varphi_{\tau}^2\left(\frac{\langle X_i, u\rangle}{\|u\|_{\Sigma}}\right)\right]\\
&\qquad\qquad+ \frac{1}{n}\sum_{i=1}^n \left\{\ell''(Y_i, X_i^{\top}\beta^*)\varphi_{\tau}^2\left(\frac{\langle X_i, u\rangle}{\|u\|_{\Sigma}}\right) - \mathbb{E}\left[\ell''(Y_i, X_i^{\top}\beta^*)\varphi_{\tau}^2\left(\frac{\langle X_i, u\rangle}{\|u\|_{\Sigma}}\right)\right]\right\}.
\end{split}
\end{equation}
We now control the first term by noting that
\begin{equation}
\begin{split}
&\frac{1}{n}\sum_{i=1}^n \mathbb{E}\left[\ell''(Y_i, X_i^{\top}\beta^*)\left\{\left(\frac{\langle X_i, u\rangle}{\|u\|_{\Sigma}}\right)^2 - \varphi_{\tau}^2\left(\frac{\langle X_i, u\rangle}{\|u\|_{\Sigma}}\right)\right\}\right]\\ 
&\qquad\le \frac{1}{n}\sum_{i=1}^n \mathbb{E}\left[\ell''(Y_i, X_i^{\top}\beta^*)\left(\frac{\langle X_i, u\rangle}{\|u\|_{\Sigma}}\right)^2\mathbbm{1}\left\{|\langle u, X_i\rangle| \ge \tau\|u\|_{\Sigma}/2\right\}\right]\\
&\qquad\le B_2\frac{1}{n}\sum_{i=1}^n \mathbb{E}\left[\frac{\langle u, X_i\rangle^2}{\|u\|_{\Sigma}^2}\mathbbm{1}\{|\langle u, X_i\rangle| \ge \tau\|u\|_{\Sigma}/2\}\right]\\
&\qquad\le B_2\int_{\tau/2}^{\infty} 2t\exp\left(-\frac{t^2}{2L^2}\right)dt = 2B_2L^2\exp\left(-\frac{\tau^2}{8L^2}\right),
\end{split}
\end{equation}
using the boundedness of $\ell''$ and $L$-sub-Gaussianity of $X_i$. 
This implies that
\begin{equation}
\begin{split}
\frac{1}{n}\sum_{i=1}^n \mathbb{E}\left[\ell''(Y_i, X_i^{\top}\beta^*)\varphi_{\tau}^2\left(\frac{\langle X_i, u\rangle}{\|u\|_{\Sigma}}\right)\right] &\ge \frac{\|u\|_{K_n}}{\|u\|_{\Sigma}} - 2B_2L^2\exp\left(-\frac{\tau^2}{8L^2}\right)dt\ge \frac{1}{2B_3^{1/2}},
\end{split}
\end{equation}
for $\tau := 2\sqrt{2}L(\log(4B_3^{1/2}B_2L^2))^{1/2}$. 

Combining this with the lower bound on $\Delta_2f_n(\beta^*, u)$, we get for all $u\in T$,
\[
\Delta_2f_n(\beta^*, u) \ge \|u\|_{\Sigma}^2\exp(-3\tau\|u\|_{\Sigma})\left[0.5B_3^{-1/2} - \tilde{Q}_n(T)\right],
\]
where
\[
\tilde{Q}_n(T) := \max_{u\in T}\left|\frac{1}{n}\sum_{i=1}^n \left\{\ell''(Y_i, X_i^{\top}\beta^*)\varphi_{\tau}^2\left(\frac{\langle X_i, u\rangle}{\|u\|_{\Sigma}}\right) - \mathbb{E}\left[\ell''(Y_i, X_i^{\top}\beta^*)\varphi_{\tau}^2\left(\frac{\langle X_i, u\rangle}{\|u\|_{\Sigma}}\right)\right]\right\}\right|.
\]
Before bounding $\tilde{Q}_n(T)$, note from assumption~\ref{eq:LossAssump} that
\[
\|u\|_K^2 \le B_2\|u\|_{\Sigma}^2 \le B_2B_3\|u\|_K,
\] 
and hence for all $u\in T$
\begin{equation}\label{eq:RSCinK}
\Delta_2f_n(\beta^*, u) \ge B_2^{-1}\|u\|_{K}^2\exp(-3B_3^{1/2}\tau\|u\|_{K})\left[0.5B_3^{-1/2} - \tilde{Q}_n(T)\right],
\end{equation}
We now bound $\tilde{Q}_n(T)$.
Define the function class
\[
F := \{(x,y)\mapsto (\ell''(y, x^{\top}\beta^*))^{1/2}\varphi_{\tau}(\langle u, x\rangle/\|u\|_{\Sigma}):\,u\in T\}.
\]
From this definition, it follows that
\[
\tilde{Q}_n(T) = \sup_{f\in F}\left|\frac{1}{n}\sum_{i=1}^n \left\{f^2(X_i, Y_i) - \mathbb{E}[f^2(X_i, Y_i)]\right\}\right|.
\]
We now apply Theorem 5.5 of~\cite{dirksen2015tail} to get with probability $1 - \exp(-t)$,
\[
\tilde{Q}_n(T) \le C\left(\frac{\gamma(F)\mbox{diam}(F)}{\sqrt{n}} + \frac{\gamma^2(F)}{n} + \frac{t^{1/2}\mbox{diam}^2(F)}{\sqrt{n}} + \frac{t\mbox{diam}^2(F)}{n}\right),
\]
for some constant $C > 0$ where
\begin{equation*}
\begin{split}
\mbox{diam}(F) &:= \sup_{f\in F}\|f(X, Y)\|_{\psi_2} \le \sup_{u\in T}\frac{\|(\ell''(Y, X^{\top}\beta^*))^{1/2}\varphi_{\tau}(\langle u, X\rangle/\|u\|_{\Sigma})\|_{\psi_2}}{\|u\|_{\Sigma}}\\ 
&\le B_2^{1/2}\sup_{u\in T}\frac{\|\langle u, X\rangle\|_{\psi_2}}{\|u\|_{\Sigma}} \le B_2^{1/2}L,
\end{split}
\end{equation*}
and
\begin{equation*}
\begin{split}
\gamma(F) &:= B_2^{1/2}\mathbb{E}\left[\max_{u\in T}\frac{|\langle \Sigma^{1/2}u, g\rangle|}{\|u\|_{\Sigma}}\right] = B_2^{1/2}\gamma(T,\Sigma).
\end{split}
\end{equation*}
Substituting these quantities in~\eqref{eq:RSCinK} for $\|u\|_K \le 1$ implies the result.
\end{proof}
The following result proves a rate result for linear and logistic regression with $h(\beta)=\lambda\|\beta\|_1$ (based on the restricted strong convexity result above). Define for the loss $\ell(\cdot,\cdot)$,
\[
f_n(\beta) := \frac{1}{n}\sum_{i=1}^n \ell(Y_i, X_i^{\top}\beta)\quad\mbox{and}\quad \hat{\beta} := \argmin_{\beta\in\mathbb{R}^p}\,f_n(\beta) + \lambda\|\beta\|_1
\]
Recall from~\Cref{lemma:penalizedLassoBelongstoCone} that if $\lambda = L\sigma^*(1 + 3\xi)\sqrt{2\log(p/s)/n}$ in case of squared loss and $\lambda = (L/2)(1 + 3\xi)\sqrt{2\log(p/s)/n}$ in case of logistic loss, for some $\xi > 0$ then on the event~\eqref{event-h-subgaussian},
\[
\hat{\beta} - \beta^* \in T_{\texttt{lasso}}(s(6+2\xi^{-1})^2) := \{\delta\in\mathbb{R}^p:\|\delta\|_1 \le \sqrt{s}\|\delta\|(6 + 2\xi^{-1})\}.
\]
This holds for both the linear and logistic Lasso case.
\propRateLogisticLasso*
Assumption~\eqref{eq:RSC} is verified in~\Cref{prop:RSCVerifiedC} and the assumptions of this proposition are satisfied for both squared and logistic loss.
Proposition above proves the rate for logistic Lasso with $\sqrt{\log(p/s)}$ instead of (usually seen) $\sqrt{\log(p)}$.
See~\cite{dedieu2018error} for a similar result
that requires more stringent conditions on $\Sigma$.

The argument in \Cref{prop:RateLogisticLasso-mainpaper} is not specific
to the Lasso. The same argument yields a similar bound for the Group-Lasso
with rate $\sqrt{sd+\log(M/s)}/n\sqrt n$ under the setting of \Cref{lemma:GroupLassoBelongstoCone}
using the Gaussian-width bound from \Cref{lemma:GroupLassoBoundGamma}.
\begin{proof}[Proof of \Cref{prop:RateLogisticLasso-mainpaper}]
Set $\delta = \hat{\beta} - \beta^*$. From the definition of $\hat{\beta}$ we get that
\begin{equation}
\begin{split}
f_n(\beta^* + \delta) + \lambda\|\beta^* + \delta\|_1 &\le f_n(\beta^*) + \lambda\|\beta^*\|_1\\
f_n(\beta^* + \delta) - f_n(\beta^*) &\le \lambda(\|\beta^*\|_1 - \|\beta^* + \delta\|_1).
\end{split}
\end{equation}
Adding $-\delta^{\top}\nabla f_n(\beta^*)$ from both sides, we get
\begin{equation}
\begin{split}
f_n(\beta^* + \delta) - f_n(\beta^*) - \delta^{\top}\nabla f_n(\beta^*) &\le -\delta^{\top}\nabla f_n(\beta^*) + \lambda(\|\beta^*\|_1 - \|\beta^* + \delta\|_1).
\end{split}
\end{equation}
Applying triangle inequality and then H\"older's inequality on $S$ (the support of $\beta^*$), we obtain
\begin{equation}
\begin{split}
f_n(\beta^* + \delta) - f_n(\beta^*) - \delta^{\top}\nabla f_n(\beta^*) &\le -\delta^{\top}\nabla f_n(\beta^*) + \lambda(\|\delta_S\|_1 - \|\delta_{S^c}\|_1)\\
&\le -\delta^{\top}\nabla f_n(\beta^*) + \lambda(\sqrt{s}\|\delta\|_2 - \|\delta_{S^c}\|_1).
\end{split}
\end{equation}
From~\Cref{lem:LemmaA1}, take $Z := n^{1/2}\Sigma^{-1/2}\nabla f_n(\beta^*)$ (this will be $\tilde{Z}$ for logistic loss). Define $U = \Sigma^{1/2}Z/(L\sigma*)$ for squared loss and
$U = 2\Sigma^{1/2}\tilde{Z}/L$ for logistic loss (as in the proof of Lemma~\ref{lemma:penalizedLassoBelongstoCone}), we get on event~\eqref{event-h-subgaussian}
\begin{equation}
\begin{split}
f_n(\beta^* + \delta) - f_n(\beta^*) - \delta^{\top}\nabla f_n(\beta^*) &\le 
\begin{cases}
{\sigma^*}{Ln^{-1/2}}\left[U^{\top}\delta + \sigma^*L^{-1}\lambda n^{1/2}\|\delta\|\sqrt{s}\right],&\mbox{for squared loss}\\
0.5{Ln^{-1/2}}\left[U^{\top}\delta + 2L^{-1}\lambda n^{1/2}\|\delta\|\sqrt{s}\right],&\mbox{for logistic loss}
\end{cases}
\end{split}
\end{equation}
We will now complete the proof for squared loss and the result for logistic loss follows by replacing $\sigma^*$ by $1/2$.
\begin{equation}\label{eq:UprBnd}
\begin{split}
f_n(\beta^* + \delta) - f_n(\beta^*) - \delta^{\top}\nabla f_n(\beta^*) &\le \sigma^*{L\|\delta\|}(2 + 5\xi)\sqrt{2s\log(p/s)/n}\\
&\le \frac{\sigma^*L\|\delta\|_K(2 + 5\xi)}{B_3^{1/2}\phi(T)}\sqrt{2s\log(p/s)/n}.
\end{split}
\end{equation}
The last inequality above follows from the fact that $\|u\|_K \ge B_3^{-1/2}\|u\|_{\Sigma} \ge B_3^{-1/2}\phi(T)^{-1}\|u\|$. Now set $t = \min\{1, \|\delta\|_K^{-1}{\sigma^*L(2 + 5\xi)\sqrt{s\log(p/s)/n}/(2B_3^{1/2}\phi(T)\theta^2)}\}$ so that $\|t\delta\|_{K} \le 1$ by assumption~\eqref{eq:AssumptionSP}. Combining~\eqref{eq:RSC} with $u = t\delta$ and~\eqref{eq:UprBnd} yields
\begin{equation}
\begin{split}
\theta^2t^2\|\delta\|_{K}^2 &\le f_n(\beta^* + u) - f_n(\beta^*) - u^{\top}\nabla f_n(\beta^*)\\
&\le t[f_n(\beta^* + \delta) - f_n(\beta^*) - \delta^{\top}\nabla f_n(\beta^*)] \le t\frac{\sigma^*L\|\delta\|_K(2 + 5\xi)}{B_3^{1/2}\phi(T)}\sqrt{2s\log(p/s)/n}.
\end{split}
\end{equation}
Rearranging this yields,
\[
t \le \frac{\sigma^*L(2 + 5\xi)}{B_3^{1/2}\phi(T)\theta^2\|\delta\|_K}\sqrt{2s\log(p/s)/n}.
\]
By definition of $t$, this inequality implies $t = 1$ and hence
\[
\frac{\sigma^*L(2 + 5\xi)}{B_3^{1/2}\phi(T)\theta^2\|\delta\|_K}\sqrt{2s\log(p/s)/n} \ge 1 \quad\Rightarrow\quad \|\delta\|_K \le \frac{\sigma^*L(2 + 5\xi)}{B_3^{1/2}\phi(T)\theta^2}\sqrt{2s\log(p/s)/n}.
\]
\end{proof}


\section[Sparsity of eta]{Proof of sparsity of $\eta$}
\label{section-sparsity-eta-appendix}
\propSparsityEta*
\begin{proof}[Proof of~\Cref{prop:SparsityEta}]
The estimator $\eta$ is defined by
\[
    \eta := \argmin_{\beta\in\mathbb{R}^p}\,\frac{1}{2}\left\|K^{1/2}\left(\beta - \beta^*\right) - n^{-1/2}Z\right\|^2 + h(\beta)
\]
where $h$ is the Group-Lasso penalty given by \eqref{h-GL} for some tuning parameter $\lambda>0$, and 
\[
Z := \frac{1}{\sqrt n}\sum_{i=1}^n K^{-1/2}X_i\ell'(Y_i,X_i^{\top}\beta^*).
\]
For the squared loss, $Z$ is also given in \eqref{Z-squared-loss}
while $Z$ for the logistic loss is given by $\tilde Z$ in \eqref{Z-tilde-logistic}.
By the KKT conditions, we get
\[
    0 \in K\left(\eta - \beta^*\right) - n^{-1/2}K^{1/2}Z + \partial h(\eta) 
\]
where $\partial h(\eta)$ is the sub-differential of $h$ at $\eta$.
This implies that for any group $k$ such that $\eta_{G_k} \neq 0$,
\begin{equation}\label{eq:KKTCondEta}
    \left\|\left(K\left(\eta - \beta^*\right) - n^{-1/2}K^{1/2}Z\right)_{G_k}\right\| = \lambda.
\end{equation}
Let $\hat{A}\subset [M] $ be the set of $s$ largest values of $\|(K^{1/2}Z)_{G_k}\|$.
Let $\mbox{supp}_G(\eta) = \{k\in[M]:\eta_{G_k}\ne 0\}$
and let $\mathcal{B}$ be a subset of $\mbox{supp}_G(\eta)$
such that $\hat{A}\cap \mathcal{B} = \emptyset$ (or equivalently $\mathcal{B} = \hat{A}^c\cap
\mbox{supp}_G(\eta)$).
Note that
\begin{equation}
\begin{split}
|\mbox{supp}_G(\eta)| &= |\mbox{supp}_G(\eta)\cap \hat{A}| + |\mbox{supp}_G(\eta)\cap \hat{A}^c| = |\mbox{supp}_G(\eta)\cap\hat{A}| + |\mathcal{B}|\\
&\le |\hat{A}| + |\mathcal{B}| = s + |\mathcal{B}|.
\end{split}
\end{equation}
Hence it is enough to bound $|\mathcal{B}|$. Set $\hat{\lambda} = n^{-1/2}\max_{k\in\hat{A}^c}\|(K^{1/2}Z)_{G_k}\|$. Summing the squares of the KKT condition~\eqref{eq:KKTCondEta} above for $j\in \mathcal{B}$ yields
\begin{equation}
\begin{split}
    |\mathcal{B}|\lambda^2 &= \sum_{k\in \mathcal{B}} \left(K\left(\eta - \beta^*\right) - n^{-1/2}K^{1/2}Z\right)_{G_k}^2\\
&= (n^{-1/2}Z - K^{1/2}(\eta - \beta^*))^{\top}\left(\sum_{k\in \mathcal{B}}\sum_{j\in G_k}K^{1/2}e_je_j^{\top}K^{1/2}\right)(n^{-1/2}Z - K^{1/2}(\eta - \beta^*))\\
&= (n^{-1/2}Z - K^{1/2}(\eta - \beta^*))^{\top}M(n^{-1/2}Z - K^{1/2}(\eta - \beta^*)),
\end{split}
\end{equation}
where $M = \sum_{k\in\mathcal B}\sum_{j\in G_k}K^{1/2}e_je_j^{\top}K^{1/2}$. Taking square root and using triangle inequality, we get
\begin{equation}\label{eq:LambdaHatLambdaBound}
\begin{split}
    \sqrt{|\mathcal{B}|}\lambda &\le \sqrt{\sum_{k\in \mathcal{B}} \|n^{-1/2}(K^{1/2}Z)_{G_k}\|^2} + \|M^{1/2}K^{1/2}(\eta - \beta^*)\|\\
&\le \sqrt{|\mathcal{B}|}\hat{\lambda} + \|M^{1/2}K^{1/2}(\eta - \beta^*)\|.
\end{split}
\end{equation}
The second inequality above follows from the fact that $\mathcal{B}\subset\hat{A}^c$
and the definition of $\hat A$.
By~\eqref{inequality-lambda-hat-GL},
with probability at least given by the right hand side of \eqref{eq:above-event-group-lasso},
$\hat\lambda\le(1+\xi)^{-1}\lambda$ and hence
\begin{equation}
    \sqrt{|\mathcal{B}|} ~\le~ (1+\xi^{-1})\|M^{1/2}K^{1/2}(\eta - \beta^*)\|/\lambda
    \label{plugin-back-group-lasso}
\end{equation}
Therefore, $\sqrt{|\mathcal{B}|} \le (1 + \xi^{-1})\|M\|^{1/2}_{op}\|K^{1/2}(\eta - \beta^*)\|/\lambda$.
By strong convexity of the quadratic program \eqref{eq:FirstOrderApproxRegularized}
(cf., e.g. \cite[Lemma 1]{bellec2018prediction}
or \cite[Lemma A.2]{bellec2016slope}) we have
$$
\tfrac 1 2 \|K^{1/2}(\eta-\beta^*)\|^2
\le
n^{-1/2}Z^T\Sigma^{1/2}(\eta-\beta^*)
+ \lambda\sum_{k=1}^M (\|\beta^*_{G_k}\| - \|\eta_{G_k}\|).
$$
On event \eqref{eq:above-event-group-lasso}, by the rightmost inequality
of \eqref{final-conclusion-lemma-group-lasso-cone} in the proof of \Cref{lemma:GroupLassoBelongstoCone} we have $\eta-\beta^*\in T$
for the set $T$ defined in \Cref{lemma:GroupLassoBoundGamma},
and by the inequalities
of \eqref{final-conclusion-lemma-group-lasso-cone}, the previous display yields
\begin{equation}
    \label{risk-bound-eta-GL}
\tfrac 1 2 \|K^{1/2}(\eta-\beta^*)\|^2
\le (3+\xi)\sqrt s \lambda \|\eta-\beta^*\|
\le (3+\xi)\sqrt s \lambda \phi(T)^{-1} B_3 \|K^{1/2}(\eta-\beta^*)\|
\end{equation}
and $\|K^{1/2}(\eta-\beta^*)\| \le 2(3+\xi)\sqrt{s}\lambda \phi(T)^{-1} B_3$.
Plugging this bound back in \eqref{plugin-back-group-lasso} we obtain
\begin{equation}
\begin{split}
\sqrt{|\mathcal{B}|} 
&\le \sqrt{s}\;
\|K_{\bar G,\bar G}\|_{op}^{1/2}
\;
2(3+\xi)(1+\xi^{-1})B_3\phi(T)^{-1} 
\end{split}
\end{equation}
where $\bar G=\cup_{k\in\mathcal{B}}G_k$.
Hence we obtain $|\mathcal B| \lesssim s$ as required
for any $\mathcal B$ such that the ratio 
$\|K_{\bar G,\bar G}\|_{op}^{1/2}/\phi(T)$ is bounded.

The proof for $\hbeta$ (in the squared loss case) follows the same argument.
The only major difference is that we have the empirical Gram matrix
$\X^T\X/n$ instead of $\Sigma$ (where $\X$ is the design matrix with rows $X_1,...,X_n$),
and we need to bound the quantities
$\|(\X^T\X/n)_{\bar G,\bar G}\|_{op}$
and $\|\X(\hbeta-\beta^*)\|/\sqrt n$.
It is enough to notice that
$\|\X(\hbeta-\beta^*)\|/\sqrt n
=\|\Sigma^{1/2}(\hbeta-\beta^*)\|(1+o(1))$
and 
$\|(\X^T\X/n)_{\bar G,\bar G}\|_{op}\le
(1+o(1))\|\Sigma_{\bar G,\bar G}\|_{op}$
by an application of \cite{plan_vershynin_liaw2017simple}
with the Gaussian-width bound given in \Cref{lemma:GroupLassoBoundGamma}.
\end{proof}


\end{document}